\documentclass[12pt,twoside]{amsart}
\usepackage{amssymb,amsmath,amsthm, amscd, 
enumerate, mathrsfs}
\usepackage{graphicx, hhline}
\usepackage[all]{xy}
\usepackage{hyperref}
\usepackage{geometry}
\title{On quasi-log schemes}
\author{Osamu Fujino}
\date{2022/4/5, version 0.06}
\address{Department of Mathematics, Graduate School of Science, 
Kyoto University, Kyoto 606-8502, Japan}
\email{fujino@math.kyoto-u.ac.jp}

\DeclareMathOperator{\Spec}{Spec}
\DeclareMathOperator{\Bs}{Bs}
\DeclareMathOperator{\Supp}{Supp}
\DeclareMathOperator{\Ngklt}{Ngklt}
\DeclareMathOperator{\Nglc}{Nglc}
\DeclareMathOperator{\Nqklt}{Nqklt}
\DeclareMathOperator{\Nqlc}{Nqlc}
\DeclareMathOperator{\Nklt}{Nklt}
\DeclareMathOperator{\Nlc}{Nlc}
\DeclareMathOperator{\NLC}{NLC}

\DeclareMathOperator{\Exc}{Exc}
\DeclareMathOperator{\Pic}{Pic}
\DeclareMathOperator{\Div}{Div}
\DeclareMathOperator{\Weil}{Weil}
\DeclareMathOperator{\Coker}{Coker}
\DeclareMathOperator{\rank}{rank}
\DeclareMathOperator{\coeff}{coeff}
\newtheorem{thm}{Theorem}[section]
\newtheorem{lem}[thm]{Lemma}

\newtheorem{cor}[thm]{Corollary}

\newtheorem*{claim-n}{Claim}

\theoremstyle{definition}
\newtheorem{step}{Step}

\newtheorem{ex}[thm]{Example}
\newtheorem{defn}[thm]{Definition}
\newtheorem{rem}[thm]{Remark}
\newtheorem*{ack}{Acknowledgments}  
\makeatletter
    
    \@addtoreset{equation}{section}
\makeatother

\begin{document}

\maketitle 

\begin{abstract}
The notion of quasi-log schemes 
was first introduced by Florin Ambro in his epoch-making 
paper:~Quasi-log varieties. 
In this paper, we establish the basepoint-free theorem of 
Reid--Fukuda type for quasi-log schemes in full generality. 
Roughly speaking, 
it means that all the results for quasi-log schemes claimed 
in Ambro's paper hold true. 
The proof is Kawamata's X-method with the aid of the theory 
of basic slc-trivial fibrations. 
For the reader's convenience, we 
make many comments on the theory of 
quasi-log schemes in order to make it more accessible. 
\end{abstract}

\tableofcontents

\section{Introduction}\label{q-sec1}

Let us start with a simple situation. Let $X$ be 
a normal projective variety defined over $\mathbb C$, 
the complex number field. 
Then the following two conditions (I) and (II) are equivalent: 
\begin{itemize}
\item[(I)] there 
exist a proper birational morphism 
$f\colon Y\to X$ from a smooth variety $Y$ and a $\mathbb Q$-divisor 
$B_Y$ on $Y$ such that 
$\Supp B_Y$ is a simple normal crossing divisor 
with $\lfloor B_Y\rfloor\leq 0$ satisfying 
\begin{itemize}
\item[(1)] $K_Y+B_Y\sim _{\mathbb Q}f^*\omega$ for some 
$\mathbb Q$-Cartier $\mathbb Q$-divisor $\omega$ on $X$,  and 
\item[(2)] the natural map 
$$
\mathcal O_X\to f_*\mathcal O_Y(\lceil -B_Y\rceil)
$$
is an isomorphism,  
\end{itemize} 
and 
\item[(II)] $(X, B)$, where $B=f_*B_Y$ such that 
$f$ and $B_Y$ are as in (I), 
is kawamata log terminal in the usual sense. 
\end{itemize} 
Even if we replace the assumption that $f$ 
is proper birational in (I) with one that $f$ 
is only proper, many results, for example, 
the cone and contraction theorem, still hold true 
for $X$ with respect to $\omega$ (see \cite{fujino1}). 
This observation plays a crucial role in \cite{fujino12} and \cite{fujino16} 
(see also Section \ref{p-sec4}). 
Hence it is natural to consider the following more general 
setting. 

Let $X$ be a scheme with an $\mathbb R$-Cartier divisor 
$\omega$ on $X$. 
Note that $X$ may be reducible and may have non-reduced 
components. Of course, $X$ is not necessarily 
equidimensional. Assume that 
there exists a proper morphism $f\colon Y\to X$ from a 
globally embedded simple normal crossing pair 
$(Y, B_Y)$ such that 
\begin{itemize}
\item[(a)] $K_Y+B_Y\sim _{\mathbb R}f^*\omega$, and 
\item[(b)] the natural map $\mathcal O_X\to f_*\mathcal O_Y(\lceil 
-(B^{<1}_Y)\rceil)$ induces an isomorphism 
$$
\mathcal I_{X_{-\infty}}\overset{\simeq}{\longrightarrow} 
f_*\mathcal O_Y(\lceil 
-(B_Y^{<1})\rceil-\lfloor B_Y^{>1}\rfloor),  
$$ 
where $\mathcal I_{X_{-\infty}}$ is the defining ideal sheaf of 
a closed subscheme $X_{-\infty}\subsetneq X$. 
\end{itemize}
Then we call 
$$
\left(X, \omega, f\colon (Y, B_Y)\to X\right)
$$ 
or simply $[X, \omega]$ a {\em{quasi-log scheme}} 
(see \cite{ambro}). 
We can prove various Kodaira type vanishing 
theorems, the cone and contraction theorem, 
and so on, for quasi-log schemes (see 
\cite[Chapter 6]{fujino23}). 
If $X_{-\infty}$ is empty, then 
we say that 
$$
\left(X, \omega, f\colon (Y, B_Y)\to X\right)
$$ 
or $[X, \omega]$ is a {\em{quasi-log canonical pair}}. 
In general, quasi-log canonical pairs are reducible and 
are not equidimensional. 
However, it is surprising that they have only Du Bois singularities 
(see \cite{fujino-liu-haidong3}). 
Note that a log canonical pair can be seen 
as a quasi-log canonical pair by considering a suitable resolution 
of singularities. Hence \cite{fujino-liu-haidong3} is a 
complete generalization of \cite{kollar-kovacs}. 
More generally, 
let $(X, B)$ be a quasi-projective semi-log canonical pair. 
Then $[X, K_X+B]$ naturally becomes a quasi-log canonical 
pair (see \cite{fujino17}). Hence any union of 
some slc strata of $(X, B)$, which is denoted by 
$W$, has a natural quasi-log canonical 
structure induced from the one on $[X, K_X+B]$ by adjunction. 
Therefore, $W$ has only Du Bois singularities, the cone and 
contraction theorem holds for $W$ with 
respect to $(K_X+B)|_W$, and 
various Kodaira type vanishing theorems can be formulated 
on $W$. 

\medskip 

One of the main purposes of this paper is to establish the 
following theorem, which is \cite[Theorem 7.2]{ambro}. 
We note that there exists no detail of the proof of 
Theorem \ref{q-thm1.1} in \cite{ambro}. 

\begin{thm}[Basepoint-free theorem of 
Reid--Fukuda type for quasi-log schemes]\label{q-thm1.1}
Let $[X, \omega]$ be a quasi-log scheme, 
let $\pi\colon X\to S$ be a {\em{proper}} morphism between schemes, 
and let $L$ be a $\pi$-nef Cartier divisor on $X$ such that 
$qL-\omega$ is nef and log big over $S$ 
with respect to $[X, \omega]$ for some positive 
real number $q$. Assume that 
$\mathcal O_{X_{-\infty}}(mL)$ is $\pi$-generated for every $m\gg 0$. 
Then $\mathcal O_X(mL)$ is $\pi$-generated 
for every $m\gg 0$. 
\end{thm}

This type of basepoint-free theorem was first considered by Reid 
in \cite[10.4.~Eventual freedom]{shokurov} in order to avoid 
Zariski's famous counterexample (see \cite[Remark 3-1-2.(2)]{kmm}). 
Note that Section 10 in \cite{shokurov} was 
written by Reid when he translated Shokurov's paper 
from Russian to English (see 
\cite[\S 10.~Commentary by M.~Reid]{shokurov}). 
Then Fukuda treated this problem in a series of papers 
(see \cite{fukuda1}, \cite{fukuda2}, and \cite{fukuda3}). 
We know that the basepoint-free theorem of Reid--Fukuda 
type holds true for divisorial log terminal pairs 
of arbitrary dimension (see \cite{fujino2}). 
In \cite{ambro}, Ambro claims Theorem \ref{q-thm1.1} 
without proof. In \cite[Section 6.9]{fujino23}, 
we proved Theorem \ref{q-thm1.1} under the extra assumption that 
$X_{-\infty}=\emptyset$ and $\pi$ is projective. 
In \cite{fujino20}, we treated Theorem \ref{q-thm1.1} 
when $\pi$ is projective and $X_{-\infty}$ may be nonempty.  
In \cite[Section 6.9]{fujino23} and \cite{fujino20}, we use 
Kodaira's lemma for big divisors. Hence the projectivity of 
$\pi$ is indispensable in \cite[Section 6.9]{fujino23} and \cite{fujino20}. 
Our approach to Theorem \ref{q-thm1.1} in this paper 
is completely different 
from the one in \cite[Section 6.9]{fujino23} and \cite{fujino20}. 
We use the theory of basic $\mathbb R$-slc-trivial fibrations, 
which is discussed in \cite[Sections 3 and 5]{fujino-hashizume2}, 
in 
order to prove 
Theorem \ref{q-thm1.1}.  
This means that we use the theory of variations of 
mixed Hodge structure 
on cohomology with compact support for the 
proof of Theorem \ref{q-thm1.1} (see \cite{fujino-fujisawa1} 
and \cite{fujino-fujisawa-saito}). 
We think that the main importance of 
Theorem \ref{q-thm1.1} 
is not in the statement but in the techniques in the proof. 
By Theorem \ref{q-thm1.1} and the author's series of 
papers, we can recover all the results for 
quasi-log schemes in \cite{ambro}. 
As an obvious corollary of Theorem \ref{q-thm1.1}, we have: 

\begin{cor}[Basepoint-free theorem of Reid--Fukuda type 
for log canonical pairs]\label{q-cor1.2}
Let $(X, \Delta)$ be a log canonical pair and 
let $\pi\colon X\to S$ be 
a proper morphism to a scheme $S$. 
Let $L$ be a $\pi$-nef Cartier divisor on $X$ such that 
$qL-(K_X+\Delta)$ is nef and log big over $S$ with 
respect to $(X, \Delta)$ for some positive real number $q$. 
Then $\mathcal O_X(mL)$ is $\pi$-generated 
for every $m\gg 0$. 
\end{cor}

When $\pi$ is projective, Corollary \ref{q-cor1.2} 
is nothing but \cite[Corollary 6.9.4]{fujino23}. 
As far as we know, there is no approach 
to Corollary \ref{q-cor1.2} without using the theory of 
quasi-log schemes. 
The main ingredient of 
the proof of Theorem \ref{q-thm1.1} is the 
following theorem, which is a generalization of 
\cite[Theorem 1.7]{fujino30} and \cite[Theorem 7.1]{fujino35}. 
As is well known, some generalizations of Kodaira's canonical 
bundle formula are very useful for various geometric problems 
(see \cite{fujino-mori}, \cite{fujino3}, \cite{fujino4}, 
\cite{fujino-gongyo1}, \cite{fujino-gongyo2}, and so on). 
Theorem \ref{q-thm1.3} below is a kind of canonical bundle formula. 
The proof depends on \cite[Theorem 5.1 and 
Corollary 5.2]{fujino-hashizume2}, 
that is, the theory of basic $\mathbb R$-slc-trivial fibrations. 
Roughly speaking, Theorem 
\ref{q-thm1.3} says that every normal irreducible quasi-log 
scheme naturally becomes a {\em{generalized pair}}. 
For the details of generalized pairs, we recommend the 
reader to see Birkar's survey article \cite{birkar}. 

\begin{thm}[Normal irreducible 
quasi-log schemes]\label{q-thm1.3}
Let $$
\left(X, \omega, f\colon (Y, B_Y)\to X\right)
$$ 
be a quasi-log scheme such that 
$X$ is a normal variety, $f$ is projective, and every 
stratum of $Y$ is dominant onto $X$. 
Then $f\colon (Y, B_Y)\to X$ is a basic $\mathbb R$-slc-trivial 
fibration. 
Let $\mathbf B$ and $\mathbf M$ be the discriminant 
and moduli $\mathbb R$-b-divisors associated to 
$f\colon (Y, B_Y)\to X$, respectively. 
Then there exists a projective 
birational morphism $p\colon X'\to X$ from a 
smooth quasi-projective variety $X'$ such that 
\begin{itemize}
\item[(i)] $\mathbf K+\mathbf B=\overline 
{\mathbf K_{X'}+\mathbf B_{X'}}$ holds, 
where $\mathbf K$ is the canonical 
b-divisor of $X$, 
\item[(ii)] $\Supp \mathbf B_{X'}$ is a simple 
normal crossing divisor on $X'$, 
\item[(iii)] $\mathbf M=\overline {\mathbf M_{X'}}$ 
holds such that $\mathbf M_{X'}$ is a potentially nef 
$\mathbb R$-divisor on $X'$, 
\item[(iv)] $p\left(\mathbf B^{\geq 1}_{X'}\right)
=\Nqklt(X, \omega)$ holds set 
theoretically, and 
\item[(v)] $p\left(\mathbf B^{> 1}_{X'}\right)
=\Nqlc(X, \omega)$ holds set 
theoretically. 
\end{itemize} 
Note that 
$$
\mathbf K+\mathbf B+\mathbf M=\overline \omega
$$ 
holds by definition. 
Moreover, if we put 
$$
\mathcal J_{\Ngklt}: =
p_*\mathcal O_{X'}\left(-\lfloor \mathbf B_{X'}\rfloor\right)
$$ 
and 
$$
\mathcal J_{\Nglc}: =p_*\mathcal O_{X'}
\left(-\lfloor \mathbf B_{X'}\rfloor+\mathbf B^{=1}_{X'}\right)
=p_*\mathcal O_{X'}\left(\lceil -(\mathbf B^{<1}_{X'})\rceil -\lfloor 
\mathbf B^{>1}_{X'}\rfloor\right), 
$$ 
then $\mathcal J_{\Ngklt}=p_*\mathcal O_{X'}\left(-\lfloor 
\mathbf B^{\geq 1}_{X'}\rfloor\right)$ and 
$\mathcal J_{\Nglc}=p_*\mathcal O_{X'}\left(-\lfloor 
\mathbf B^{>1}_{X'}\rfloor\right)$ are ideal sheaves 
on $X$ such that the following inclusions 
$$
\mathcal J_{\Ngklt}\subset \mathcal I_{\Nqklt(X, \omega)} 
\quad \text{and} 
\quad \quad
\mathcal J_{\Nglc}\subset \mathcal I_{\Nqlc(X, \omega)}
$$ 
hold, where $\mathcal I_{\Nqklt(X, \omega)}$ and 
$\mathcal I_{\Nqlc(X, \omega)}$ are the defining 
ideal sheaves of $\Nqklt(X, \omega)$ and 
$\Nqlc(X, \omega)$, respectively. 

We note that in the above statement 
$\mathbf B$ and $\mathbf M$ become 
$\mathbb Q$-b-divisors when 
$$
\left(X, \omega, f\colon (Y, B_Y)\to X\right)
$$ 
has a $\mathbb Q$-structure. 
\end{thm}

We will prove Theorem \ref{q-thm1.1} by 
combining Kawamata's X-method with Theorem \ref{q-thm1.3} 
in the framework of quasi-log schemes. 
We note that we do not use the minimal model program in this paper. 

Let us explain the idea of the proof of Theorem \ref{q-thm1.1}. 
For simplicity of notation, we assume that $S$ is a point. 
By the standard argument in the theory of quasi-log schemes, 
we can reduce the problem to the case where $X$ is irreducible and 
$\mathcal O_{\Nqklt(X, \omega)}(mL)$ is generated 
by global sections for every $m\gg 0$. 
This implies that $\Bs|mL|\cap \Nqklt(X, \omega)=\emptyset$ 
for every $m\gg 0$. Let $\nu\colon Z\to X$ be the normalization. 
Then $[Z, \nu^*\omega]$ has a natural quasi-log scheme structure 
with 
$\nu_*\mathcal I_{\Nqklt(Z, \nu^*\omega)}=\mathcal I_{\Nqklt(X, \omega)}$. 
Moreover, we may assume that 
$[Z, \nu^*\omega]$ satisfies Theorem \ref{q-thm1.3}. 
By the classical X-method, we can prove that there are many 
global sections of 
$\mathcal O_Z(m\nu^*L)\otimes \mathcal J_{\Ngklt}$ for $m\gg 0$, 
where 
$\mathcal J_{\Ngklt}$ is the ideal sheaf defined in Theorem \ref{q-thm1.3}. 
By $\mathcal J_{\Ngklt}\subset \mathcal I_{\Nqklt(Z, \nu^*\omega)}$ and 
$\nu_*\mathcal I_{\Nqklt(Z, \nu^*\omega)}=\mathcal I_{\Nqklt(X, \omega)}$, 
$$H^0\left(Z, \mathcal O_Z(m\nu^*L)\otimes \mathcal J_{\Ngklt}\right)
\subset H^0\left(X, \mathcal O_X(mL)
\otimes \mathcal I_{\Nqklt(X, \omega)}\right)$$ 
holds for every $m$. 
Hence $\mathcal O_X(mL)\otimes \mathcal I_{\Nqklt(X, \omega)}$ has 
many global sections such that 
$\Bs|mL|\cap \left(X\setminus \Nqklt(X, \omega)\right)=\emptyset$. 
Therefore, we obtain $\Bs|mL|=\emptyset$ for $m \gg 0$. 

\medskip 

Finally, we note: 

\begin{rem}\label{q-rem1.4} 
Recently, the minimal model program 
for threefolds in positive and mixed 
characteristic is developing rapidly. Moreover, 
the minimal model program for K\"ahler threefolds 
is studied extensively. Unfortunately, however, 
our framework 
of quasi-log schemes only works for algebraic 
varieties in characteristic zero. 
This is because it heavily depends 
on the theory of mixed Hodge structures 
and the theory of variations of mixed 
Hodge structure. It is a challenging 
and interesting problem to discuss 
the theory of quasi-log schemes in other settings. 
\end{rem}

This paper is organized as follows. 
In Section \ref{q-sec2}, we make some comments on base 
fields and the Lefschetz principle for the reader's convenience. 
In Section \ref{p-sec3}, we collect some basic definitions. 
In Section \ref{p-sec4}, we slightly reformulate the 
Kawamata--Shokurov basepoint-free theorem. 
The results in this section can be proved 
by Kawamata's X-method without 
difficulties. 
In Section \ref{q-sec5}, we quickly recall the definition of 
quasi-log schemes and basic slc-trivial fibrations 
and explain some fundamental results. 
In Section \ref{q-sec6}, we prove Theorem \ref{q-thm1.3}, which is 
one of the main results of this paper. 
Section \ref{q-sec7} is devoted to the proof of Theorem \ref{q-thm1.1}. 
In Section \ref{q-sec8}, we make many comments on \cite{ambro} 
to help the reader 
understand differences between Ambro's original 
approach in \cite{ambro} and our framework of quasi-log schemes. 

\begin{ack}
The author was partially 
supported by JSPS KAKENHI Grant Numbers 
JP19H01787, JP20H00111, 
JP21H00974, JP21H04994.
He would like to thank 
Kenta Hashizume very much for fruitful discussions. 
Finally, he thanks the referee for many useful comments. 
\end{ack}

We will work over $\mathbb C$, the complex number field, 
throughout this paper. 
A {\em{scheme}} means a separated scheme of finite type over $\mathbb C$. 
A {\em{variety}} means an integral scheme, that is, 
an irreducible and reduced separated scheme of finite type over $\mathbb C$. 
We will freely use the framework of quasi-log schemes 
established in \cite[Chapter 6]{fujino23}. 
We note that $\mathbb Z$, $\mathbb Q$, and 
$\mathbb R$ denote the set of {\em{intergers}}, 
{\em{rational numbers}}, and {\em{real numbers}}, respectively. 
We also note that $\mathbb Q_{>0}$ and $\mathbb R_{>0}$ 
are the set of {\em{positive rational numbers}} and 
{\em{positive real numbers}}, respectively. 
In this paper, 
the expression \lq ... for every $m\gg 0$\rq \ 
means that \lq there exists a positive number $m_0$ such that ... 
for every $m\geq m_0$.\rq 

\section{On the Lefschetz principle}\label{q-sec2}

In this short section, before starting the main 
contents of this paper, we make some comments 
on base fields and the Lefschetz principle for the reader's convenience.

\begin{rem}[On the base field $k$]\label{q-rem2.1}
We mainly work over $\mathbb C$, the complex number field, in the 
papers on quasi-log schemes (see \cite[Chapter 6]{fujino23}). 
This is because the author's approach 
depends on the theory of mixed Hodge structures on cohomology with 
compact support. However, almost all results on quasi-log schemes 
hold true over any algebraically closed field $k$ of characteristic zero. 
For example, we can prove the vanishing theorems for quasi-log 
schemes over $k$ by the Lefschetz principle. 
Hence the cone and contraction theorem for quasi-log schemes 
defined over $k$ 
can be proved as an application of some vanishing theorems. 
When we treat {\em{sufficiently general fibers}}, 
{\em{uniruledness}}, {\em{rationally chain connectedness}}, 
and so on, we have to take care of the base field $k$ if 
the cardinality of $k$ is countable (see \cite{fujino35}). 
It is obvious that 
some results, for example, the simply connectedness of quasi-log 
canonical Fano pairs, make sense only over $\mathbb C$ 
(see \cite{fujino-liu-wenfei} and \cite{fujino31}). 
We note that we can check that all the results obtained 
in this paper hold true over any algebraically closed 
field $k$ of characteristic zero without any difficulties. 
\end{rem}

Let us quickly see how to use the Lefschetz principle. 
Let $X$ be a projective scheme defined over an algebraically closed 
field $k$ of characteristic zero. 
Let $L$ be a Cartier divisor on $X$ and let $H$ be an ample 
Cartier divisor on $X$. 
We can take a subfield $k_0$ of $k$, 
which is finitely generated over $\mathbb Q$, and 
a scheme $X_0$ defined over $k_0$, a Cartier divisor 
$L_0$ on $X_0$, and an ample 
Cartier divisor $H_0$ on $X_0$ such that 
$X\simeq X_0\times _{\Spec k_0} \Spec k$, 
$L\simeq L_0\times _{\Spec k_0}\Spec k$, and 
$H\simeq H_0\times _{\Spec k_0}\Spec k$. 
We consider some embedding $k_0\subset \mathbb C$ and 
the induced morphism $\Spec \mathbb C\to \Spec k_0$. 
Then we put 
$X_{\mathbb C}:= X_0\times _{\Spec k_0} \Spec \mathbb C$, 
$L_{\mathbb C}:=L_0\times _{\Spec k_0}\Spec \mathbb C$, and 
$H_{\mathbb C}:=H_0\times _{\Spec k_0}\Spec \mathbb C$. 
Since $H$ is ample, $H_0$ and $H_{\mathbb C}$ are both 
ample. 
Note that $L$, $L_0$, and $L_{\mathbb C}$ are nef 
if and only if $L+rH$, $L_0+rH_0$, 
and $L_{\mathbb C}+rH_{\mathbb C}$ are 
ample for every rational number $r$ with $0<r\ll 1$, respectively. 
Hence, $L$ is nef if and only if $L_{\mathbb C}$ is nef. 

\section{Preliminaries}\label{p-sec3}

Here, we collect some basic definitions for the reader's convenience. 
Let $X$ be a scheme and let $\Pic(X)$ be the group of 
line bundles on $X$, that is, the {\em{Picard group}} of $X$. 
An element of $\Pic(X)\otimes_{\mathbb Z}\mathbb R$ 
(resp.~$\Pic(X)\otimes _{\mathbb Z}\mathbb Q$) is called an 
{\em{$\mathbb R$-line bundle}} (resp.~a {\em{$\mathbb Q$-line 
bundle}}) on $X$. We write the group law of $\Pic(X)\otimes 
_{\mathbb Z}\mathbb Q$ additively for simplicity of notation. 
Let $\Div(X)$ be the group of 
Cartier divisors on $X$. 
An element of $\Div(X)\otimes _{\mathbb Z} \mathbb R$ 
(resp.~$\Div(X)\otimes _{\mathbb Z} \mathbb Q$) is called an 
{\em{$\mathbb R$-Cartier divisor}} (resp.~a 
{\em{$\mathbb Q$-Cartier divisor}}) on $X$. 
Let $\Delta_1$ and $\Delta_2$ be $\mathbb R$-Cartier 
(resp.~$\mathbb Q$-Cartier) divisors on $X$. 
Then $\Delta_1\sim _{\mathbb R} \Delta_2$ (resp.~$\Delta_1\sim 
_{\mathbb Q}\Delta_2$) means 
that $\Delta_1$ is {\em{$\mathbb R$-linearly}} 
(resp.~{\em{$\mathbb Q$-linearly}}) {\em{equivalent}} to $\Delta_2$. 
There exists a natural group homomorphism 
$
\Div (X)\to \Pic(X)
$ given by $A\mapsto \mathcal O_X(A)$, 
where $A$ is a Cartier divisor on $X$. 
It induces a homomorphism 
$
\delta_X\colon \Div(X)\otimes _{\mathbb Z}\mathbb R 
\to \Pic(X)\otimes _{\mathbb Z}\mathbb R$.  
We sometimes write 
$
A+\mathcal L\sim _{\mathbb R} B+\mathcal M
$ 
for $A, B\in \Div(X)\otimes _{\mathbb Z} \mathbb R$ and $\mathcal L, 
\mathcal M\in \Pic(X)\otimes _{\mathbb Z} \mathbb R$. 
This means that 
$
\delta_X(A)+\mathcal L=\delta_X(B)+\mathcal M
$ 
holds in $\Pic(X)\otimes _{\mathbb Z}\mathbb R$. 
We usually use this type of abuse of notation, that is, 
the confusion of $\mathbb R$-line bundles with 
$\mathbb R$-Cartier divisors. 
In the theory of minimal models for higher-dimensional algebraic 
varieties, we sometimes use $\mathbb R$-Cartier divisors for 
ease of notation even when they should be $\mathbb R$-line bundles.  

\medskip 

Let us recall the definition of {\em{potentially nef divisors}}. 
We need it in Theorem \ref{q-thm1.3}. 

\begin{defn}[{Potentially nef divisors, see 
\cite[Definition 2.5]{fujino30}}]\label{p-def3.1} 
Let $X$ be a normal 
variety and let $D$ be a divisor on $X$. 
If there exist a completion $X^\flat$ of $X$, 
that is, $X^\flat$ is a complete normal 
variety and contains $X$ as a dense Zariski open subset, and 
a nef divisor $D^\flat$ on $X^\flat$ such that 
$D=D^\flat|_X$, then $D$ is called 
a {\em{potentially nef}} divisor on $X$. 
A finite $\mathbb Q_{>0}$-linear (resp.~$\mathbb R_{>0}$-linear) 
combination of potentially nef divisors is called 
a {\em{potentially nef}} 
$\mathbb Q$-divisor (resp.~$\mathbb R$-divisor). 
\end{defn}

\begin{rem}\label{p-rem3.2} 
(i) Let $D$ be a nef $\mathbb R$-divisor on a smooth 
projective variety $X$. 
Then $D$ is not necessarily a potentially nef $\mathbb R$-divisor. 
This means that $D$ is not always a finite $\mathbb R_{>0}$-linear 
combination of nef Cartier divisors on $X$. 
(ii) Let $X$ be a normal variety and let $D$ be a potentially 
nef $\mathbb R$-divisor on $X$. Then $D\cdot C\geq 0$ for every 
projective curve $C$ on $X$. 
In particular, $D$ is $\pi$-nef for every proper 
morphism $\pi\colon X\to S$ to a scheme $S$. 
\end{rem}

It is convenient to use {\em{b-divisors}} to explain 
several results. We note that the notion of b-divisors 
was first introduced by Shokurov. 
Let us recall the definition of b-divisors for the reader's 
convenience. 

\begin{defn}[b-divisors]\label{p-def3.3} 
Let $X$ be a normal variety and let $\Weil (X)$ be the space of Weil 
divisors on $X$. A {\em{b-divisor}} on $X$ is an element: 
$$
\mathbf D\in \mathbf{Weil} (X) :=\lim_{Y\to X} \Weil (Y), 
$$ 
where the (projective) limit is taken over all proper birational 
morphism $f\colon Y\to X$ from a normal variety $Y$ under the 
pushforward homomorphism $f_*: \Weil (Y)\to \Weil (X)$. 
We can define {\em{$\mathbb Q$-b-divisors}} and 
{\em{$\mathbb R$-b-divisors}} on $X$ similarly.  
If $\mathbf D=\sum d_\Gamma \Gamma$ is an $\mathbb R$-b-divisor 
on a normal variety $X$ and $f\colon Y\to X$ is a proper birational morphism 
from a normal variety $Y$, then the {\em{trace}} of $\mathbf D$ on $Y$ is 
the $\mathbb R$-divisor 
$$
\mathbf D_Y:=\sum_{\text{$\Gamma$ is a divisor on $Y$}}
d_\Gamma \Gamma. 
$$
\end{defn}

\begin{defn}[Canonical b-divisors]\label{p-def3.4}
Let $X$ be a normal variety and let 
$\omega$ be a top rational differential 
form of $X$. 
Then $(\omega)$ defines a 
b-divisor $\mathbf K$. We call $\mathbf K$ 
the {\em{canonical b-divisor}} of $X$. 
\end{defn}

\begin{defn}[$\mathbb R$-Cartier closures]\label{p-def3.5}
The {\em{$\mathbb R$-Cartier 
closure}} of an $\mathbb R$-Cartier $\mathbb R$-divisor 
$D$ on a normal variety $X$ is the $\mathbb R$-b-divisor 
$\overline D$ with trace 
$$
\overline D _Y=f^*D, 
$$
where $f \colon Y\to X$ is a proper birational morphism 
from a normal variety $Y$. 
\end{defn}

\begin{defn}[{\cite[Definition 2.12]{fujino30}}]\label{p-def3.6}
Let $X$ be a normal variety. 
An $\mathbb R$-b-divisor $\mathbf D$ of $X$ 
is {\em{b-potentially nef}} 
(resp.~{\em{b-semi-ample}}) if there 
exists a proper birational morphism $X'\to X$ from a normal 
variety $X'$ such that $\mathbf D=\overline {\mathbf D_{X'}}$, that 
is, $\mathbf D$ is the $\mathbb R$-Cartier 
closure of $\mathbf D_{X'}$, and that 
$\mathbf D_{X'}$ is potentially nef 
(resp.~semi-ample). 
An $\mathbb R$-b-divisor $\mathbf D$ 
of $X$ is {\em{$\mathbb R$-b-Cartier}} 
if there is a proper birational morphism $X'\to X$ from a normal 
variety $X'$ such that $\mathbf D=\overline{\mathbf D_{X'}}$. 
\end{defn}

\begin{defn}\label{p-def3.7}
Let $X$ be an equidimensional reduced scheme. 
Note that $X$ is not necessarily regular in codimension one. 
Let $D$ be an $\mathbb R$-divisor (resp.~a $\mathbb Q$-divisor), 
that is, 
$D$ is a finite formal sum $\sum _i d_iD_i$, where 
$D_i$ is an irreducible reduced closed subscheme of $X$ of 
pure codimension one and $d_i\in \mathbb R$ 
(resp.~$d_i\in \mathbb Q$) for every $i$ 
such that $D_i\ne D_j$ for $i\ne j$. 
We put 
\begin{equation*}
D^{<1} =\sum _{d_i<1}d_iD_i, \quad 
D^{= 1}=\sum _{d_i= 1} D_i, \quad 
D^{>1} =\sum _{d_i>1}d_iD_i, \quad \text{and} \quad 
\lfloor D\rfloor =\sum _i \lfloor d_i \rfloor D_i, 
\end{equation*}
where 
$\lfloor d_i\rfloor$ is the integer defined by $d_i-1< 
\lfloor d_i\rfloor \leq d_i$. 
We note that $\lceil D\rceil =-\lfloor -D\rfloor$ and 
$\{D\}=D-\lfloor D\rfloor$. Similarly, we put 
$$
D^{\geq 1}=\sum _{d_i\geq 1}d_i D_i. 
$$

Let $D$ be an $\mathbb R$-divisor (resp.~a $\mathbb Q$-divisor) 
as above. 
We call $D$ a {\em{subboundary}} $\mathbb R$-divisor 
(resp.~$\mathbb Q$-divisor) 
if $D=D^{\leq 1}$ holds. 
When $D$ is effective and $D=D^{\leq 1}$ holds, 
we call $D$ a {\em{boundary}} 
$\mathbb R$-divisor (resp.~$\mathbb Q$-divisor). 
\end{defn}

The following definition is 
standard and is well known. 

\begin{defn}[Singularities of pairs]\label{p-def3.8}
Let $X$ be a variety and let $E$ be a prime divisor on $Y$ 
for some proper birational
morphism $f\colon Y\to X$ from a normal variety $Y$. 
Then $E$ is called a divisor {\em{over}} $X$. 
A {\em{normal pair}} $(X, \Delta)$ consists of 
a normal variety $X$ and an $\mathbb R$-divisor $\Delta$ on $X$ 
such that $K_X+\Delta$ is $\mathbb R$-Cartier. 
Let $(X, \Delta)$ be a normal pair and let 
$f\colon Y\to X$ be a proper 
birational morphism from a normal variety $Y$. 
Then we can write 
$$
K_Y=f^*(K_X+\Delta)+\sum _E a(E, X, \Delta)E
$$ 
with 
$$f_*\left(\underset{E}\sum a(E, X, \Delta)E\right)=-\Delta, 
$$ 
where $E$ runs over prime divisors on $Y$. 
We call $a(E, X, \Delta)$ the {\em{discrepancy}} of $E$ with 
respect to $(X, \Delta)$. 
Note that we can define the discrepancy $a(E, X, \Delta)$ for 
any prime divisor $E$ over $X$ by taking a suitable 
resolution of singularities of $X$. 
If $a(E, X, \Delta)\geq -1$ (resp.~$>-1$) for 
every prime divisor $E$ over $X$, 
then $(X, \Delta)$ is called {\em{sub log canonical}} (resp.~{\em{sub 
kawamata log terminal}}). 
We further assume that $\Delta$ is effective. 
Then $(X, \Delta)$ is 
called {\em{log canonical}} 
and {\em{kawamata log terminal}} 
if it is sub log canonical and sub kawamata log terminal, respectively. 

Let $(X, \Delta)$ be a log canonical pair. If there 
exists a projective birational morphism 
$f\colon Y\to X$ from a smooth variety $Y$ such that 
both $\Exc(f)$, 
the exceptional locus of $f$,  
and  $\Exc(f)\cup \Supp f^{-1}_*\Delta$ are simple 
normal crossing divisors on $Y$ and that 
$a(E, X, \Delta)>-1$ holds for every 
$f$-exceptional divisor $E$ on $Y$, 
then $(X, \Delta)$ is 
called {\em{divisorial log terminal}} ({\em{dlt}}, for short). 
\end{defn}

\begin{defn}[Log canonical centers, non-lc loci, and so on]\label{p-def3.9}
Let $(X,\Delta)$ be a normal pair. 
If there exist a projective birational morphism 
$f\colon Y\to X$ from a normal variety $Y$ and a prime divisor $E$ on $Y$ 
such that $(X, \Delta)$ is 
sub log canonical in a neighborhood of the 
generic point of $f(E)$ and that 
$a(E, X, \Delta)=-1$, then $f(E)$ is called a 
{\em{log canonical center}} of 
$(X, \Delta)$. 

From now on, we further assume that 
$\Delta$ is effective. 
Let $f\colon Y\to X$ be a resolution with
$$
K_Y+\Delta_Y=f^*(K_X+\Delta)
$$
such that $\Supp \Delta_Y$ is a simple normal crossing divisor on $Y$. 
We put
$$
\mathcal J(X, \Delta)=f_*\mathcal O_Y(-\lfloor \Delta_Y\rfloor). 
$$
Then $\mathcal J(X, \Delta)$ is an ideal sheaf on $X$ 
and is known as the {\em{multiplier ideal sheaf}} 
associated to the pair $(X, \Delta)$. 
It is independent 
of the resolution $f\colon Y\to X$. The closed subscheme $\Nklt(X, \Delta)$ 
defined by $\mathcal J(X, \Delta)$ is called 
the {\em{non-klt locus}} of $(X, \Delta)$. 
It is obvious that $(X, \Delta)$ is kawamata log terminal 
if and only if $\mathcal J(X, \Delta)=\mathcal O_X$. 
Similarly, we put 
$$\mathcal J_{\NLC}(X, \Delta)=f_*\mathcal O_X(-\lfloor 
\Delta_Y\rfloor+\Delta^{=1}_Y)$$
and call it the {\em{non-lc ideal sheaf}} associated to the pair $(X, \Delta)$. 
We can check that it is independent of the resolution $f\colon Y\to X$. The
closed subscheme $\Nlc(X, \Delta)$ defined by 
$\mathcal J_{\NLC}(X, \Delta)$ 
is called the {\em{non-lc locus}} of 
$(X, \Delta)$. It is obvious that $(X, \Delta)$ is log canonical if and only if 
$\mathcal J_{\NLC}(X, \Delta)=\mathcal O_X$. 
By definition, the natural inclusion 
$$
\mathcal J(X, \Delta)\subset \mathcal J_{\NLC}(X, \Delta)
$$ 
always holds. Therefore, 
we have 
$$
\Nlc(X, \Delta)\subset \Nklt(X, \Delta). 
$$
\end{defn}

For the details of $\mathcal J_{\NLC}(X, \Delta)$, 
we recommend the reader to see \cite{fujino9} and \cite{fujino14}. 

\section{Classical basepoint-free theorems}\label{p-sec4}

In this section, we will reformulate some classical results 
in order to apply them to the proof of Theorem \ref{q-thm1.1}. 
Everything in this section can be proved by the X-method. 
For the details of the X-method, see \cite{kmm},  
\cite{kollar-mori}, and \cite{matsuki}. 
We note that this section is similar to \cite[Section 2]{fujino16}. 

\medskip 

Let us start with Shokurov's nonvanishing theorem. 

\begin{thm}[Shokurov's nonvanishing theorem]\label{p-thm4.1} 
Let $X$ be a smooth variety and let $B$ be an $\mathbb R$-divisor 
on $X$ such that 
$\Supp B$ is a simple normal crossing divisor on $X$ with 
$\lfloor B\rfloor \leq 0$. 
Let $\pi\colon X\to S$ be a proper morphism 
to a scheme $S$ and let $D$ be a $\pi$-nef Cartier divisor on $X$. 
Assume that $aD-(K_X+B)$ is nef and big over $S$ for 
some positive integer $a$. 
Then $\pi_*\mathcal O_X(mD+\lceil -B\rceil)\ne 0$ for 
every $m\gg 0$. 
\end{thm}

\begin{proof}[Sketch of Proof]
By taking the Stein factorization and considering a sufficiently general 
fiber of $\pi\colon X\to S$, we may assume that $S$ is a point. 
By taking a resolution of singularities of $X$ and using 
Kodaira's lemma, we may further assume that 
$X$ is projective, $aD-(K_X+B)$ is ample, and 
$B$ is a $\mathbb Q$-divisor. 
In this case, the statement is well known. 
For the details, see \cite{kmm}, 
\cite{kollar-mori}, and \cite{matsuki}.  
\end{proof}

For the proof of Theorem \ref{q-thm1.1}, the following formulation of 
the Kawamata--Shokurov basepoint-free theorem is very useful. 

\begin{thm}[Kawamata--Shokurov 
basepoint-free theorem]\label{p-thm4.2} 
Let $X$ be a smooth variety and let $B$ be an $\mathbb R$-divisor 
on $X$ such that 
$\Supp B$ is a simple normal crossing divisor on $X$ with 
$\lfloor B\rfloor \leq 0$. 
Let $\pi\colon X\to S$ be a proper morphism 
to a scheme $S$ and let $D$ be a $\pi$-nef Cartier divisor on $X$. 
Assume the following conditions: 
\begin{itemize}
\item[(1)] $aD-(K_X+B)$ is nef and big over $S$ for some positive 
integer $a$, and 
\item[(2)] there exists a positive integer $k$ such that the 
natural inclusion 
\begin{equation*}
\pi_*\mathcal O_X(lD)\hookrightarrow \pi_*\mathcal O_X(lD+\lceil -B\rceil)
\end{equation*} 
is an isomorphism for every $l\geq k$. 
\end{itemize}
Then $\mathcal O_X(mD)$ is $\pi$-generated for every $m\gg 0$. 
\end{thm}

We give a detailed proof for the sake of completeness. 
We note that the proof of the Kawamata--Shokurov 
basepoint-free theorem, which is 
now usually called the X-method, works 
without any changes 
(see \cite{kmm}, \cite{kollar-mori}, and 
\cite{matsuki}), although our treatment looks slightly 
different from the original one. 

\begin{proof}[Proof of Theorem \ref{p-thm4.2}]
We will use Shokurov's nonvanishing theorem 
(see Theorem \ref{p-thm4.1}), the Kawamata--Viehweg vanishing 
theorem, and Hironaka's resolution of singularities. 
\setcounter{step}{0}
 \begin{step}\label{p-step4.2.1}
 By (1) and Theorem \ref{p-thm4.1}, 
 $\pi_*\mathcal O_X(mD+\lceil -B\rceil)\ne 0$ for 
 every $m\gg 0$. 
 Hence, by (2), $\pi_*\mathcal O_X(mD)=
 \pi_*\mathcal O_X(mD+\lceil -B\rceil)\ne 0$ holds for 
 every $m\gg 0$. 
 Let $\ell$ be any prime number. 
 Then $\pi_*\mathcal O_X(\ell ^{n_0}D)=
 \pi_*\mathcal O_X(\ell^{n_0}D+\lceil -B\rceil)\ne 0$ for 
 some sufficiently large positive integer $n_0$. 
 \end{step}
 \begin{step}\label{p-step4.2.2} 
 Let $f\colon X'\to X$ be a projective birational morphism 
 from a smooth variety $X'$ such that 
 $K_{X'}+B'=f^*(K_X+B)$ and that 
 $\Supp B'$ is a simple normal crossing divisor on $X'$. 
 Then we can check that 
 $f_*\mathcal O_{X'}(\lceil -B'\rceil )=\mathcal O_X(\lceil 
 -B\rceil)$. 
 Hence, we can replace $(X, B)$ and $D$ with 
 $(X', B')$ and $f^*D$, respectively. 
 Therefore, we may assume that 
 there exists a simple normal crossing divisor 
 $G$ on $X$ such that 
$G=\sum _j F_j$ is the irreducible decomposition with 
the following properties: 
\begin{itemize}
\item the support of $G+\Supp B$ is contained in a 
simple normal crossing divisor on $X$, 
\item $\ell ^{n_0}D=L+\sum _j r_j F_j$ 
for some nonnegative integers $r_j$ 
and a $\pi$-free Cartier divisor $L$ on $X$ such that 
$$
\pi_*\mathcal O_X(\ell^{n_0}D)=\pi_*\mathcal O_X(L), 
$$ 
and 
\item $aD-(K_X+B)-\sum _j p_j F_j$ is ample over 
$S$ for suitable $0<p_j\ll 1$. 
\end{itemize}
\end{step}
 \begin{step}\label{p-step4.2.3} 
 We perturb $p_j$ suitably and choose $c>0$ 
such that 
\begin{equation*}
\left(B+c\sum _j r_j F_j+\sum _j p_j F_j\right)^{>1}=0
\end{equation*}
and 
\begin{equation*}
\left(B+c\sum _j r_j F_j+\sum _j p_j F_j\right)^{=1}
\end{equation*} 
is 
a prime divisor $F$ on $X$. 
We put 
\begin{equation*}
F+B'=B+c\sum _j r_j F_j+\sum _j p_j F_j. 
\end{equation*}
Then, by construction, we see that $\lfloor F+B'\rfloor=F+\lfloor B'\rfloor$ 
and that $0\leq \lceil -B'\rceil \leq \lceil -B\rceil$ holds. 
 \end{step}
 \begin{step}\label{p-step4.2.4}
Let $n_1$ be a positive integer such that 
$\ell^{n_1}\geq c\ell ^{n_0}+a$. 
We consider an $\mathbb R$-divisor 
\begin{equation*}
\begin{split}
N(\ell ^{n_1})&:=\ell ^{n_1}D-(K_X+F+B')\\ 
&= (\ell^{n_1}-c\ell^{n_0}-a)D+cL+\left(aD-(K_X+B)-\sum _j p_j F_j\right). 
\end{split} 
\end{equation*} 
Hence, $N(\ell^{n_1})$ is ample over $S$. 
By the Kawamata--Viehweg vanishing theorem, 
$$
R^1\pi_*\mathcal O_X(K_X+\lceil N(\ell^{n_1})\rceil)=
R^1\pi_*\mathcal O_X(\ell^{n_1}D-F-\lfloor B'\rfloor)=0. 
$$ 
Therefore, the restriction map 
\begin{equation}\label{p-eq4.1}
\pi_*\mathcal O_X(\ell^{n_1}D)=
\pi_*\mathcal O_X(\ell^{n_1}D-\lfloor B'\rfloor)
\to \pi_*\mathcal O_F(\ell^{n_1}D+\lceil -B'|_F\rceil)
\end{equation}
is surjective, where the equality 
in \eqref{p-eq4.1} follows from (2) and 
\begin{equation*}
\pi_*\mathcal O_X(\ell^{n_1}D)\hookrightarrow 
\pi_*\mathcal O_X(\ell^{n_1}D-\lfloor B'\rfloor)
\hookrightarrow 
\pi_*\mathcal O_X(\ell^{n_1}D-\lfloor B\rfloor). 
\end{equation*}
By construction, 
\begin{equation*}
\begin{split}
N(\ell^{n_1})|_F&=\ell^{n_1}D|_F-(K_X+F+B')|_F\\ 
&=\ell^{n_1}D|_F-(K_F+B'|_F)
\end{split} 
\end{equation*} 
is ample over $S$, $\Supp B'|_F$ is a simple 
normal crossing divisor on $F$, 
and $\lfloor B'|_F\rfloor\leq 0$ holds. 
By Theorem \ref{p-thm4.1}, 
we obtain that 
$\pi_*\mathcal O_F(\ell^{n_1}D+\lceil -B'|_F\rceil)\ne 0$ for 
every $n_1\gg 0$. 
Thus we have 
\begin{equation}\label{p-eq4.2}
F\not \subset \Supp \left(
\Coker \left(\pi^*\pi_*\mathcal O_X(\ell^{n_1}D)
\to \mathcal O_X(\ell^{n_1}D)\right)\right). 
\end{equation} 
Without loss of generality, we may assume that 
$n_1\geq n_0$. 
Hence, \eqref{p-eq4.2} implies that 
\begin{equation*}
\begin{split}
&\Supp 
\left(\Coker \left(\pi^*\pi_*\mathcal O_X(\ell^{n_1}D)
\to \mathcal O_X(\ell^{n_1}D)\right) \right)
\\ &\subsetneq 
\Supp \left(
\Coker \left(\pi^*\pi_*\mathcal O_X(\ell^{n_0}D)
\to \mathcal O_X(\ell^{n_0}D)\right)
\right)
\end{split}
\end{equation*}
since $F\subset \Supp \left(
\Coker \left(\pi^*\pi_*\mathcal O_X(\ell^{n_0}D)
\to \mathcal O_X(\ell^{n_0}D)\right)
\right)$ by construction. 
By Noetherian induction, 
\begin{equation*}
\Supp \left(
\Coker \left(\pi^*\pi_*\mathcal O_X(\ell^nD)
\to \mathcal O_X(\ell^nD)\right)\right)=\emptyset, 
\end{equation*} 
that is, $\mathcal O_X(\ell ^nD)$ is $\pi$-generated, 
for every $n\gg 0$. 
 \end{step}
 \begin{step}\label{p-step4.2.5} 
 We take another prime number $\ell'$. 
 Then, by Step \ref{p-step4.2.4}, 
 $\mathcal O_X({\ell'}^{n'}D)$ is $\pi$-generated for 
 every $n'\gg 0$. 
We may 
 assume that 
$\ell^n<\ell'^{n'}$ holds by taking $\ell$, $\ell'$, $n$, and $n'$ suitably. 
Note that $\gcd (\ell^n, \ell'^{n'})=1$ since 
$\ell \ne \ell'$. 
We put $m_0=
\ell^n \left(\ell'^{n'} -\left\lceil \frac{\ell'^{n'}}{\ell^n}\right\rceil\right)$. 
By Lemma \ref{p-lem4.3} below, for every positive integer $m$ with 
$m\geq m_0$, there exist nonnegative integers $u$ and $v$ such that 
$m=u\ell^n+v\ell'^{n'}$. 
This implies that $\mathcal O_X(mD)$ is $\pi$-generated for 
every $m\geq m_0$. 
 \end{step}
 We finish the proof. 
\end{proof}

We have already used the following easy lemma in 
the proof of Theorem \ref{p-thm4.2}. 
We give a proof for the sake of completeness.  

\begin{lem}\label{p-lem4.3} 
Let $a$ and $b$ be positive integers with $1<a<b$ such 
that $\gcd (a, b)=1$. Then, for any positive integer $m$ with 
$m\geq a\left(b-\left\lceil \frac{b}{a}\right\rceil\right)$, 
there exist nonnegative integers $u$ and $v$ such that 
$m=ua+vb$. 
\end{lem}

\begin{proof}
We can uniquely write $m=qa+r$ such that 
$q$ and $r$ are integers with 
$q\geq b-\left\lceil \frac{b}{a}\right\rceil$ and 
$0\leq r\leq a-1$. If $r=0$, then it is sufficient to put $u=q$ and $v=0$. 
From now on, we assume $r\ne 0$. 
Then there exists a positive integer $c$ such that 
$cb=\left\lfloor \frac{cb}{a}\right\rfloor a+r$ with 
$1\leq c\leq a-1$. 
Hence $m=\left(q-\left\lfloor \frac{cb}{a}\right\rfloor\right)a+cb$. 
Note that $$q-\left\lfloor 
\frac{cb}{a}\right\rfloor\geq b-\left\lceil \frac{b}{a}\right
\rceil-\left\lfloor \frac{(a-1)b}{a}\right\rfloor=0. 
$$ 
Thus it is sufficient to put $u=q-\left\lfloor \frac{cb}{a}\right\rfloor$ and 
$v=c$. 
\end{proof}

We need a somewhat artificial lemma for the proof of 
Theorem \ref{q-thm1.1}. 

\begin{lem}\label{p-lem4.4}
Let $\pi\colon X\to S$ be a proper morphism from a normal 
variety $X$ to an affine scheme $S$ and let 
$D$ be a $\pi$-nef Cartier divisor on $X$. 
Let $p\colon Z\to X$ be a proper birational morphism 
from a smooth variety $Z$ and let $B$ be an $\mathbb R$-divisor 
on $Z$ such that $\Supp B$ is a simple normal crossing divisor, 
$B^{<0}$ is $p$-exceptional, and $B^{\geq 1}\ne 0$.  
Assume that $ap^*D-(K_Z+B)$ is nef and big over 
$S$ for some positive real number $a$ and 
that $\Bs |mD|\cap p(B^{\geq 1})=\emptyset$ for some 
positive integer $m$, where $\Bs|mD|$ denotes the 
base locus of $|mD|$ with the reduced scheme structure. 
Then there exists a positive integer $s$ such that 
\begin{equation}\label{p-eq4.3}
H^0\left(X, \mathcal O_X(m^sD)\otimes 
p_*\mathcal O_Z(-\lfloor B\rfloor)\right)\otimes \mathcal O_X\to 
\mathcal O_X(m^sD)
\end{equation}
is surjective on $X\setminus p(B^{\geq 1})$. 
Note that $p_*\mathcal O_Z(-\lfloor B\rfloor)$ is an ideal 
sheaf on $X$. 
In particular, $\mathcal O_X(m^sD)$ is $\pi$-generated. 
\end{lem}

\begin{proof} The well-known X-method works with some 
minor modifications. 
\setcounter{step}{0}
\begin{step}\label{p-step4.4.1}
Let $q\colon Z'\to Z$ be a projective birational morphism 
from a smooth quasi-projective variety $Z'$ such that 
$K_{Z'}+B_{Z'}=q^*(K_Z+B)$ and that $\Supp B_{Z'}$ is a 
simple normal crossing divisor on $Z'$. 
It is easy to see that $q_*\mathcal O_{Z'}(-\lfloor B_{Z'}\rfloor)
=\mathcal O_Z(-\lfloor B\rfloor)$. 
Hence we can replace $(Z, B)$ and $p\colon Z\to X$ with $(Z', B_{Z'})$ 
and $p\circ q\colon Z'\to X$, respectively. 
Therefore, we may assume that 
there exists a simple normal crossing divisor $G$ on $Z$ such that 
$G=\sum _j F_j$ is the irreducible decomposition with 
the following properties: 
\begin{itemize}
\item[(1)] the support of $G+\Supp B$ is contained in a 
simple normal crossing divisor on $Z$, 
\item[(2)] $p^*mD=L+\sum _j r_j F_j$ for some nonnegative integers $r_j$ 
and a $\pi\circ p$-free Cartier divisor $L$ on $Z$ such that 
$$
H^0(X, \mathcal O_X(mD))=H^0(Z, \mathcal O_Z(L))
$$ 
and that $\Bs |mD|=\bigcup _{r_j>0} p(F_j)$, and 
\item[(3)] $ap^*D-(K_Z+B)-\sum _j p_j F_j$ is ample over 
$S$ for suitable $0<p_j\ll 1$. 
\end{itemize}
\end{step}
\begin{step}\label{p-step4.4.2}
By the usual argument in Kawamata's X-method, 
we can perturb $p_j$ suitably and choose $c>0$ 
such that $$\left(B+c\sum _j r_j F_j+\sum _j p_j F_j\right)^{>1}=0$$ 
on $Z\setminus \Supp B^{\geq 1}$ and 
$$\left(B+c\sum _j r_j F_j+\sum _j p_j F_j\right)^{=1}$$ is 
a prime divisor $F$ on $Z\setminus \Supp B^{\geq 1}$. 
We note that if $r_j>0$ then $F_j$ is disjoint from 
$\Supp B^{\geq 1}$ by the assumption $\Bs|mD|\cap p(B^{\geq 1})=
\emptyset$. 
We put 
$$F+B'=B+c\sum _j r_j F_j+\sum _j p_j F_j. 
$$ 
Then, by construction, $\lfloor F+B'\rfloor=F+\lfloor B'\rfloor$, 
$\Supp (B')^{\geq 1}=\Supp B^{\geq 1}$, 
and $p_*\mathcal O_Z(-\lfloor B'\rfloor)\subset 
p_*\mathcal O_Z(-\lfloor B\rfloor)$. 
Note that $F$ is disjoint from $\Supp (B')^{\geq 1}$ by construction. 
\end{step}
\begin{step}\label{p-step4.4.3}
Let $b$ be a positive integer such that 
$b\geq cm+a$. 
We consider an $\mathbb R$-divisor 
\begin{equation*}
\begin{split}
N(b)&:=bp^*D-(K_Z+F+B')\\ 
&= (b-cm-a)p^*D+cL+\left(ap^*D-(K_Z+B)-\sum _j p_j F_j\right). 
\end{split} 
\end{equation*} 
Therefore, we see that $N(b)$ is ample over $S$. 
By the Kawamata--Viehweg vanishing theorem, 
$$
H^1(Z, \mathcal O_Z(K_Z+\lceil N(b)\rceil))=
H^1(Z, \mathcal O_Z(bp^*D-F-\lfloor B'\rfloor))=0. 
$$ 
Hence the restriction map 
\begin{equation}\label{p-eq4.4}
H^0(Z, \mathcal O_Z(bp^*D-\lfloor B'\rfloor))
\to H^0(F, \mathcal O_F(bp^*D+\lceil -B'|_F\rceil))
\end{equation}
is surjective. 
Note that 
\begin{equation*}
\begin{split}
N(b)|_F&=bp^*D|_F-(K_Z+F+B')|_F\\ 
&=bp^*D|_F-(K_F+B'|_F)
\end{split} 
\end{equation*} 
is ample over $S$. 
By construction, $\Supp B'|_F$ is a simple 
normal crossing divisor on $F$ and $\lfloor B'|_F\rfloor\leq 0$ holds. 
By Theorem \ref{p-thm4.1}, 
we obtain that 
$H^0(F, \mathcal O_F(bp^*D+\lceil -B'|_F\rceil))\ne 0$ for 
every $b\gg 0$. 
Therefore, $H^0(Z, \mathcal O_Z(bp^*D-\lfloor B'\rfloor))$ 
has a section which does not vanish on 
the generic point of $F$ for every $b\gg 0$ 
by the surjection \eqref{p-eq4.4}. 
This means that there exists some positive integer $t$ such that 
\begin{equation*}
H^0\left(X, \mathcal O_X(m^tD)\otimes 
p_*\mathcal O_Z(-\lfloor B\rfloor)\right)\otimes \mathcal O_X\to 
\mathcal O_X(m^tD)
\end{equation*}
is surjective at the generic point of $p(F)$. 
It is obvious that 
$\Bs |m^t D|\cap p(B^{\geq 1})=\emptyset$ holds. 
Hence we can apply the same argument as above 
to $m^tD$. 
Then, by Noetherian induction, we can find a sufficiently 
large and divisible positive integer $s$ such that the map in \eqref{p-eq4.3} 
is surjective. 
\end{step}
\begin{step}\label{p-step4.4.4}
By \eqref{p-eq4.3}, 
$\Bs|m^sD|\cap (X\setminus p(B^{\geq 1}))=\emptyset$. 
On the other hand, $\Bs|mD|\cap p(B^{\geq 1})=\emptyset$ by 
assumption. Hence $\Bs|m^sD|\cap p(B^{\geq 1})=\emptyset$. 
Therefore, we obtain $\Bs|m^sD|=\emptyset$. This means 
that $\mathcal O_X(m^sD)$ is $\pi$-generated. 
\end{step}
We finish the proof of Lemma \ref{p-lem4.4}. 
\end{proof}

We close this section with a remark on the X-method. 

\begin{rem}\label{p-rem4.5}
As is well known, the X-method works very well for kawamata 
log terminal pairs. Precisely speaking, 
the notion of kawamata log terminal pairs was introduced 
to make the X-method work well. 
Unfortunately, however, it is not so powerful for log canonical 
pairs. Hence we proposed a new approach in order 
to prove the cone and contraction theorem for log 
canonical pairs (see \cite{fujino13} and \cite{fujino14}). 
Our approach is based on some vanishing theorems obtained 
by the theory of mixed Hodge structures on cohomology 
with compact support 
(see \cite{fujino7}, \cite{fujino22}, \cite[Chapter 5]{fujino23}, 
and \cite{fujino27}). 
In some sense, it is simpler than the X-method. 
It is mysterious that the proof of Theorem \ref{q-thm1.1} 
given in this paper needs the X-method.   
\end{rem}

\section{Quasi-log schemes and basic slc-trivial fibrations}\label{q-sec5}

In this section, let us quickly review the theory of quasi-log schemes 
and the framework of basic slc-trivial fibrations. 

Let $Y$ be a simple normal crossing divisor 
on a smooth 
variety $M$ and let $B$ be an $\mathbb R$-divisor 
on $M$ such that 
$\Supp (B+Y)$ is a simple normal crossing divisor on $M$ and that 
$B$ and $Y$ have no common irreducible components. 
We put $B_Y=B|_Y$ and consider the pair $(Y, B_Y)$. 
We call $(Y, B_Y)$ a {\em{globally embedded simple normal 
crossing pair}} and $M$ the {\em{ambient space}} 
of $(Y, B_Y)$. A {\em{stratum}} of $(Y, B_Y)$ is a log canonical  
center of $(M, Y+B)$ that is contained in $Y$. 

\medskip 

Let us recall the definition of {\em{quasi-log schemes}}. 

\begin{defn}[{Quasi-log 
schemes, see \cite[Definition 6.2.2]{fujino23}}]\label{q-def5.1}
A {\em{quasi-log scheme}} is a scheme $X$ with an 
$\mathbb R$-Cartier divisor 
(or $\mathbb R$-line bundle) 
$\omega$ on $X$, a closed subscheme 
$X_{-\infty}\subsetneq X$, and a finite collection $\{C\}$ of 
subvarieties of $X$ such that there exists a 
proper morphism $f\colon (Y, B_Y)\to X$ from a globally 
embedded simple 
normal crossing pair satisfying the following properties: 
\begin{itemize}
\item[(1)] $f^*\omega\sim_{\mathbb R}K_Y+B_Y$ holds,  
\item[(2)] the natural map 
$\mathcal O_X
\to f_*\mathcal O_Y(\lceil -(B_Y^{<1})\rceil)$ 
induces an isomorphism 
$$
\mathcal I_{X_{-\infty}}\overset{\simeq}{\longrightarrow} 
f_*\mathcal O_Y(\lceil 
-(B_Y^{<1})\rceil-\lfloor B_Y^{>1}\rfloor),  
$$ 
where $\mathcal I_{X_{-\infty}}$ is the defining ideal sheaf of 
$X_{-\infty}$, and 
\item[(3)] the collection of subvarieties 
$\{C\}$ coincides with the images 
of the strata of $(Y, B_Y)$ that are not included in $X_{-\infty}$. 
\end{itemize}
If there is no risk of confusion, 
then we simply write $[X, \omega]$ to denote 
the above data 
$$
\left(X, \omega, f\colon (Y, B_Y)\to X\right). 
$$ 
The subvarieties $C$ 
are called the {\em{qlc strata}} of $[X, \omega]$ 
and $X_{-\infty}$ is called the {\em{non-qlc locus}} 
of $[X, \omega]$. 
Note that a qlc stratum $C$ of $[X, \omega]$ is an 
irreducible and reduced closed subscheme of $X$. 
We usually use $\Nqlc(X, 
\omega)$ 
to denote 
$X_{-\infty}$. 
If a qlc stratum $C$ of $[X, \omega]$ is not an 
irreducible component of $X$, then 
it is called a {\em{qlc center}} of $[X, \omega]$. 
The union of $X_{-\infty}$ with all 
qlc centers of $[X, \omega]$ is denoted by 
$\Nqklt(X, \omega)$. 

We say that $(X, \omega, f\colon (Y, B_Y)\to X)$ or $[X, \omega]$ 
has a {\em{$\mathbb Q$-structure}} 
if $B_Y$ is a $\mathbb Q$-divisor, $\omega$ is a $\mathbb Q$-Cartier 
divisor 
(or $\mathbb Q$-line bundle), and 
$f^*\omega\sim _{\mathbb Q} K_Y+B_Y$ holds in 
the above definition. 
\end{defn}

We need the notion of nef and log big $\mathbb R$-divisors 
in order to formulate the basepoint-free theorem of 
Reid--Fukuda type. 

\begin{defn}[Nef and log bigness]\label{q-def5.2}
Let $L$ be an $\mathbb R$-Cartier divisor 
(or $\mathbb R$-line bundle) on a quasi-log scheme $[X, \omega]$ 
and let $\pi\colon X\to S$ be a proper 
morphism between schemes. 
Then $L$ is said to be {\em{nef and log big over $S$ 
with respect to $[X, \omega]$}} 
if $L$ is $\pi$-nef and $L|_{W}$ is 
$\pi$-big for every qlc stratum $W$ of $[X, \omega]$. 
\end{defn}

The following example is very important. 
By this example, we can apply the results for quasi-log schemes to 
normal pairs. 

\begin{ex}\label{q-ex5.3}
Let $X$ be a normal variety and let $B$ be an effective 
$\mathbb R$-divisor on $X$ such that 
$K_X+B$ is $\mathbb R$-Cartier. 
Let $f\colon Y\to X$ be a proper birational morphism 
from a smooth variety $Y$ with $K_Y+B_Y=f^*(K_X+B)$ 
such that $\Supp B_Y$ is a simple normal crossing divisor on $Y$. 
Then $$
\left(X, K_X+B, f\colon (Y, B_Y)\to X\right)
$$ 
naturally becomes a quasi-log scheme. 
Note that $W$ is a qlc center of $[X, K_X+B]$ if and 
only if $W$ is a log canonical center of $(X, B)$. 
By construction, $\Nqlc(X, K_X+B)=\Nlc(X, B)$. 
Hence $(X, B)$ is log canonical if and only if 
$\Nqlc(X, K_X+B)=\emptyset$. 
\end{ex}

A pair $(X, B)$ which is Zariski locally isomorphic to 
a globally embedded simple normal crossing pair at any point is 
called a {\em{simple normal crossing pair}}. 
Let $(X, B)$ be a simple normal crossing pair and let 
$\nu\colon X^\nu \to X$ be the normalization. 
We define $B^\nu$ by $K_{X^\nu}+B^\nu=\nu^*(K_X+B)$, that is, 
$B^\nu$ is the sum of the inverse images of $B$ and the singular locus of $X$. 
We note that $X^\nu$ is a disjoint union of smooth varieties and 
$\Supp B^\nu$ is a simple normal crossing divisor on $X^\nu$. 
Then we say that $W$ is a {\em{stratum}} of $(X, B)$ 
if and only if $W$ is an irreducible component 
of $X$ or the $\nu$-image of some log canonical 
center of $(X^\nu, B^\nu)$. 
Let $(X, B)$ be a simple normal crossing pair as above and let 
$X=\bigcup_{i\in I} X_i$ be the irreducible decomposition of $X$. 
Then a {\em{stratum}} of $X$ means an irreducible 
component of $X_{i_1}\cap \cdots \cap X_{i_k}$ for 
some $\{i_1, \ldots, i_k\}\subset I$. 
It is not difficult to see that $W$ is a stratum of $X$ if and 
only if $W$ is a stratum of $(X, 0)$. 
By definition, 
it is obvious that a globally embedded simple 
normal crossing pair is a simple normal crossing 
pair. 

\medskip 

Let us recall the definition of {\em{basic slc-trivial fibrations}}. 

\begin{defn}[Basic slc-trivial fibrations]\label{q-def5.4}
A {\em{basic $\mathbb Q$-slc-trivial 
{\em{(}}resp.~$\mathbb R$-slc-trivial{\em{)}} fibration}} 
$f \colon (X, B)\to Y$ consists of 
a projective surjective morphism 
$f \colon X\to Y$ and a simple normal crossing pair $(X, B)$ satisfying 
the following properties: 
\begin{itemize}
\item[(1)] $Y$ is a normal variety,   
\item[(2)] every stratum of $X$ is dominant onto $Y$ and 
$f_*\mathcal O_X\simeq \mathcal O_Y$, 
\item[(3)] $B$ is a $\mathbb Q$-divisor 
(resp.~an $\mathbb R$-divisor) 
such that $B=B^{\leq 1}$ holds 
over 
the generic point of $Y$, 
\item[(4)] there exists 
a $\mathbb Q$-Cartier $\mathbb Q$-divisor 
(resp.~an $\mathbb R$-Cartier $\mathbb R$-divisor) 
$D$ on $Y$ such that 
$
K_X+B\sim _{\mathbb Q}f^*D 
$ (resp.~$K_X+B\sim _{\mathbb R} f^*D$), 
and   
\item[(5)] $\rank f_*\mathcal O_X(\lceil -(B^{<1})\rceil)=1$.  
\end{itemize}
\end{defn}

If there is no danger of confusion, we simply use 
basic slc-trivial fibrations to
denote basic $\mathbb Q$-slc-trivial fibrations or 
basic $\mathbb R$-slc-trivial fibrations. 
Let $f \colon (X, B)\to Y$ be a basic 
slc-trivial fibration as in Definition \ref{q-def5.4} 
and let $\nu \colon X^\nu\to X$ be the normalization with 
$K_{X^\nu}+B^\nu=\nu^*(K_X+B)$ as before. 
Let $P$ be a prime divisor on $Y$. 
By shrinking $Y$ around the generic point of $P$, 
we assume that $P$ is a Cartier divisor. We set 
$$
b_P=\max \left\{t \in \mathbb R\, \left|\, 
\begin{array}{l}  {\text{$(X^\nu, B^\nu+t\nu^*f^*P)$ is sub log canonical}}\\
{\text{over the generic point of $P$}} 
\end{array}\right. \right\}  
$$ and 
put $$
B_Y=\sum _P (1-b_P)P, 
$$ 
where $P$ runs over prime divisors on $Y$. 
Then $B_Y$ is a well-defined $\mathbb R$-divisor on 
$Y$ and is called the {\em{discriminant 
$\mathbb R$-divisor}} of $f \colon (X, B)\to Y$. We set 
$$
M_Y=D-K_Y-B_Y
$$ 
and call $M_Y$ the {\em{moduli $\mathbb R$-divisor}} of $f \colon 
(X, B)\to Y$. 
The discriminant $\mathbb R$-divisor $B_Y$ is uniquely 
determined by $f\colon X\to Y$ and $B$ geometrically. 
On the other hand, the moduli $\mathbb R$-divisor 
$M_Y$ depends on the choice of $K_X$, $K_Y$, and $D$. 
By definition, we have 
$$
K_X+B\sim _{\mathbb R}f^*(K_Y+B_Y+M_Y). 
$$
Let $\sigma\colon Y'\to Y$ be a proper birational morphism 
from a normal variety $Y'$. 
Then we have the following commutative diagram: 
$$
\xymatrix{
   (X^\sharp, B^\sharp) \ar[r]^{\mu} \ar[d] 
   & (X^\nu, B^\nu)\ar[d]^{f\circ \nu} \\
   Y' \ar[r]_{\sigma} & Y, 
} 
$$ 
where $X^\sharp$ denotes the normalization of the main 
components of $X\times _Y Y'$ 
and $B^\sharp$ is defined by $K_{X^\sharp}+B^\sharp=
\mu^*(K_{X^\nu}+B^\nu)$. 
As above, we can define $\mathbb R$-divisors 
$B_{Y'}$, $K_{Y'}$ and $M_{Y'}$ for $(X^\sharp, B^\sharp)\to Y'$ 
such that 
$\sigma^*D=K_{Y'}+B_{Y'}+M_{Y'}$, 
$\sigma_*B_{Y'}=B_Y$, $\sigma _*K_{Y'}=K_Y$ 
and $\sigma_*M_{Y'}=M_Y$ hold. 
Hence there exist a unique $\mathbb R$-b-divisor $\mathbf B$ 
such that 
$\mathbf B_{Y'}=B_{Y'}$ for every 
$\sigma \colon Y'\to Y$ and a unique 
$\mathbb R$-b-divisor $\mathbf M$ 
such that $\mathbf M_{Y'}=M_{Y'}$ for 
every $\sigma \colon Y'\to Y$. 
We call $\mathbf B$ the 
{\em{discriminant $\mathbb R$-b-divisor}} associated to 
$f \colon (X, B)\to Y$. 
The $\mathbb R$-b-divisor $\mathbf M$ is called 
the {\em{moduli $\mathbb R$-b-divisor}} associated 
to $f \colon (X, B)\to Y$. 
When $f\colon (X, B)\to Y$ is a basic $\mathbb Q$-slc-trivial fibration, 
it is easy to see that $\mathbf B$ and $\mathbf M$ 
are $\mathbb Q$-b-divisors by construction. 

By using the theory of mixed Hodge structures on 
cohomology with compact support (see \cite[Chapter 5]{fujino23}), we have: 

\begin{thm}[{\cite[Theorem 6.3.5]{fujino23}}]\label{q-thm5.5}
Let $[X, \omega]$ be a quasi-log scheme 
and let $X'$ be the union of $X_{-\infty}$ with 
a {\em{(}}possibly empty{\em{)}} union of some qlc strata of $[X, \omega]$. 
Then we have the following properties. 
\begin{itemize}
\item[(i)] {\em{(Adjunction for quasi-log schemes).}}~Assume that 
$X'\ne X_{-\infty}$. Then $X'$ naturally becomes 
a quasi-log scheme with 
$\omega'=\omega|_{X'}$ and $X'_{-\infty}=X_{-\infty}$. 
Moreover, the qlc strata of $[X', \omega']$ 
are exactly the qlc strata of $[X, \omega]$ that are included in $X'$. 
\item[(ii)] {\em{(Vanishing theorem for 
quasi-log schemes).}}~Assume that 
$\pi\colon X\to S$ is a proper morphism 
between schemes. Let $L$ be a Cartier divisor on $X$ such that 
$L-\omega$ is nef and log big over $S$ with 
respect to $[X, \omega]$. 
Then $R^i\pi_*(\mathcal I_{X'}\otimes \mathcal O_X(L))=0$ for 
every $i>0$, 
where $\mathcal I_{X'}$ is the defining ideal sheaf of $X'$ on $X$. 
\end{itemize}
\end{thm}

Theorem \ref{q-thm5.5} is a key result in the theory of quasi-log 
schemes. Note that Theorem \ref{q-thm5.5} (ii) can be seen 
as 
a Kawamata--Viehweg--Nadel vanishing theorem for 
quasi-log schemes. 

\begin{rem}\label{q-rem5.6} 
If $\Nqklt(X, \omega)\ne \Nqlc(X, \omega)$, then 
$[\Nqklt(X, \omega), \omega|_{\Nqklt(X, \omega)}]$ naturally 
becomes a quasi-log scheme by adjunction (see Theorem 
\ref{q-thm5.5} (i)). 
\end{rem}

By using the theory of variations of mixed Hodge 
structure on cohomology with compact support 
(see \cite{fujino-fujisawa1} and 
\cite{fujino-fujisawa-saito}), we have: 

\begin{thm}[{\cite[Theorem 5.1]{fujino-hashizume2}}]\label{q-thm5.7}
Let $f\colon (X, B)\to Y$ be a 
basic $\mathbb R$-slc-trivial fibration such 
that $Y$ is a smooth quasi-projective variety. 
We write 
$K_X+B\sim _{\mathbb R} f^*D$. 
Assume that 
there exists a simple normal crossing divisor 
$\Sigma$ on $Y$ such that 
$\Supp D\subset \Sigma$ and 
that every stratum of $(X, \Supp B)$ is 
smooth over $Y\setminus \Sigma$. 
Let $\mathbf B$ and $\mathbf M$ be 
the discriminant and moduli $\mathbb R$-b-divisors 
associated to $f\colon (X, B)\to Y$, respectively. 
Then 
\begin{itemize}
\item[(i)] $\mathbf K +\mathbf B=
\overline {\mathbf K_Y+\mathbf B_Y}$ holds, 
where $\mathbf K$ is the canonical b-divisor 
of $Y$, and 
\item[(ii)] $\mathbf M_Y$ is a potentially 
nef $\mathbb R$-divisor 
on $Y$ with $\mathbf M=\overline {\mathbf M_Y}$. 
\end{itemize}
\end{thm}

Theorem \ref{q-thm5.7} is the most fundamental property of 
basic slc-trivial 
fibrations. In the proof of Theorem 
\ref{q-thm1.3}, we will use the following result, which 
easily follows from Theorem \ref{q-thm5.7}.  

\begin{cor}[{\cite[Corollary 5.2]{fujino-hashizume2}}]\label{q-cor5.8} 
Let $f\colon (X, B)\to Y$ be a basic 
$\mathbb R$-slc-trivial fibration and 
let $\mathbf B$ and $\mathbf M$ be the 
discriminant and moduli $\mathbb R$-b-divisors associated 
to $f\colon (X, B)\to Y$, respectively. 
Then we have the following properties: 
\begin{itemize}
\item[(i)] $\mathbf K+\mathbf B$ is $\mathbb R$-b-Cartier, 
where $\mathbf K$ is the canonical 
b-divisor of $Y$, and 
\item[(ii)] $\mathbf M$ is b-potentially nef, that is,  
there exists a proper birational morphism $\sigma\colon Y'\to Y$ 
from a normal variety $Y'$ such that 
$\mathbf M_{Y'}$ is a potentially nef $\mathbb R$-divisor on $Y'$ and 
that $\mathbf M=\overline{\mathbf M_{Y'}}$ holds. 
\end{itemize}
\end{cor}

\begin{rem}\label{q-rem5.9}
It is conjectured that $\mathbf M_Y$ is semi-ample in Theorem 
\ref{q-thm5.7}. Unfortunately, however, it is still widely open. 
We note that $\mathbf M_Y$ is known to be semi-ample when 
$Y$ is a curve (see \cite{fujino-fujisawa-liu} and \cite{fujino-hashizume2}). 
We do not mention them here although there are several cases 
in which the semi-ampleness of $\mathbf M_Y$ is known 
when X is irreducible. 
In this paper, we are mainly interested in the case where $X$ is reducible. 
\end{rem}

Very roughly speaking, in the author's opinion, 
the theory of quasi-log schemes is a powerful framework 
to use mixed Hodge structures on cohomology with 
compact support for the study of higher-dimensional 
algebraic varieties and the theory of basic slc-trivial fibrations 
was constructed in order to make the theory of variations of 
mixed Hodge structure on cohomology with compact support 
applicable for some geometric problems.

Finally, we note that Theorems \ref{q-thm5.5}, \ref{q-thm5.7}, 
and Corollary \ref{q-cor5.8} hold 
true over any algebraically closed field $k$ of characteristic zero 
(see Section \ref{q-sec2}). 

\section{On normal irreducible quasi-log schemes}\label{q-sec6}

In this section, we will prove Theorem \ref{q-thm1.3}. 
For the proof of Theorem \ref{q-thm1.3}, 
we prepare an elementary lemma. 

\begin{lem}\label{q-lem6.1}
Let $m$ be any positive integer and let $a$ be 
any nonnegative real number. 
If $t\leq -\frac{a}{m}$, then 
the following inequality 
\begin{equation}\label{q-eq6.1}
m\left(1+\lfloor -t\rfloor\right)\geq 1+\lfloor a\rfloor
\end{equation}
holds. 
\end{lem}

\begin{proof}
We can uniquely write 
$$
a=mk+l
$$ 
for some nonnegative integer $k$ with $0\leq l<m$. 
Then 
\begin{equation*}
\begin{split}
m(1+\lfloor -t\rfloor)-(1+\lfloor a\rfloor) &\geq 
m\left(1+\left\lfloor \frac{a}{m}\right
\rfloor\right)-(1+\lfloor a\rfloor)\\ 
&\geq m(1+k)-(1+mk+m-1)\\
&=0
\end{split}
\end{equation*} 
holds. 
This implies the desired inequality. 
\end{proof}

Let us prove Theorem \ref{q-thm1.3}. 

\begin{proof}[Proof of Theorem \ref{q-thm1.3}]
We divide the proof into several small steps. 
The arguments in 
Steps \ref{q-step1.3.1} and \ref{q-step1.3.2} 
are standard. 
Step \ref{q-step1.3.3} is new. 
\setcounter{step}{0}
\begin{step}\label{q-step1.3.1}
By definition, $\mathcal I_{X_{-\infty}}=f_*\mathcal O_Y(\lceil 
-(B^{<1}_Y)\rceil -\lfloor B^{>1}_Y\rfloor)$ is an ideal sheaf on $X$ 
and 
$\Supp \lfloor B^{>1}_Y\rfloor$ is not dominant onto $X$ by $f$. 
Therefore, $\rank f_*\mathcal O_Y(\lceil -(B^{<1}_Y)\rceil)=1$ holds. 
In particular, we have $\rank f_*\mathcal O_Y=1$. 
Let $f\colon Y\to Z\to X$ be the Stein factorization of $f\colon Y\to X$. 
Since every irreducible component of $Y$ is dominant 
onto $X$ and $\rank f_*\mathcal O_Y=1$, 
$Z\to X$ is a finite birational morphism from a 
variety $Z$ onto a normal variety 
$X$. 
Then, by Zariski's main theorem, $Z\to X$ is an isomorphism. 
This means that the natural map $\mathcal O_X\to f_*\mathcal O_Y$ is 
an isomorphism. Hence $f\colon (Y, B_Y)\to X$ is a basic 
$\mathbb R$-slc-trivial fibration. 
By Corollary \ref{q-cor5.8}, we can take 
a proper birational morphism $p\colon X'\to X$ 
from a normal variety $X'$ with $\mathbf K+\mathbf B=
\overline{\mathbf K_{X'}+\mathbf B_{X'}}$ and 
$\mathbf M=\overline{\mathbf M_{X'}}$, where 
$\mathbf M_{X'}$ is a potentially nef $\mathbb R$-divisor 
on $X'$. 
By using Hironaka's resolution, 
we may further assume that $X'$ is a smooth 
quasi-projective variety and 
that $\Supp \mathbf B_{X'}$ is a simple 
normal crossing divisor on $X'$. Therefore, 
we obtain a projective birational morphism 
$p\colon X'\to X$ from a smooth 
quasi-projective variety $X'$ satisfying 
(i), (ii), and (iii). 
By the proof of \cite[Lemma 11.2]{fujino30}, 
we see that $\mathbf B_X$ is effective. 
By the argument in Step 3 in the proof of 
\cite[Theorem 7.1]{fujino35}, we can make $p\colon X'\to X$ satisfy 
(iv). We note that we can directly apply 
the argument in Step 3 in the proof of 
\cite[Theorem 7.1]{fujino35} to the basic 
$\mathbb R$-slc-trivial fibration $f\colon (Y, B_Y)\to X$ by 
Corollary \ref{q-cor5.8}. 
Moreover, by the same argument, we can check 
(v). 
\end{step}
\begin{step}\label{q-step1.3.2} 
In this step, we will check that $\mathcal J_{\Ngklt}$ and 
$\mathcal J_{\Nglc}$ are well-defined 
ideal sheaves on $X$, that is, they are independent 
of $p\colon X'\to X$. 

Let $q\colon X''\to X'$ be a projective birational morphism 
from a smooth quasi-projective variety $X''$ such that 
$\mathbf K_{X''}+\mathbf B_{X''}=q^*(\mathbf K_{X'}+\mathbf B_{X'})$ and 
that $\Supp \mathbf B_{X''}$ is a simple 
normal crossing divisor on $X''$. 
Since $(X', \{\mathbf B_{X'}\})$ is kawamata log terminal, 
$q^*\lfloor \mathbf B_{X'}\rfloor -\lfloor \mathbf B_{X''}\rfloor$ is an 
effective $q$-exceptional divisor on $X''$. 
Hence we have $$q_*\mathcal O_{X''}(-\lfloor \mathbf B_{X''}\rfloor)
=\mathcal O_{X'}(-\lfloor \mathbf B_{X'}\rfloor)$$ by projection formula. 
By this fact and Hironaka's resolution of singularities, 
we can easily see that $\mathcal J_{\Ngklt}$ is independent of 
$p\colon X'\to X$. Since $\mathbf B_X$ is effective by construction 
(see the proof of \cite[Lemma 11.2]{fujino30}), 
$\Supp \mathbf B^{<0}_{X'}$ is $p$-exceptional. 
Hence $\mathcal J_{\Ngklt}$ is an ideal sheaf on $X$ 
and $\mathcal J_{\Ngklt}=p_*\mathcal O_{X'}
\left(-\lfloor \mathbf B^{\geq 1}_{X'}\rfloor\right)$ holds. 
Of course, $\mathcal J_{\Ngklt}$ is a generalization of 
well-known multiplier 
ideal sheaves (see \cite{lazarsfeld}). 
Similarly, since $(X', \{\mathbf B_{X'}\}+\mathbf B^{=1}_{X'})$ is 
divisorial 
log terminal, $q^*\left(\lfloor \mathbf B_{X'}\rfloor -\mathbf B^{=1}_{X'}\right)-
\left(\lfloor \mathbf B_{X''}\rfloor-\mathbf B^{=1}_{X''}\right)$ is 
an effective $q$-exceptional divisor on $X''$. 
Hence $$q_*\mathcal O_{X''}\left(-\lfloor 
\mathbf B_{X''}\rfloor +\mathbf B^{=1}_{X''}\right)=
\mathcal O_{X'}\left(-\lfloor 
\mathbf B_{X'}\rfloor +\mathbf B^{=1}_{X'}\right). 
$$ 
This means that 
$\mathcal J_{\Nglc}$ is independent of $p\colon X'\to X$ by Hironaka's 
resolution of singularities. Since $\lceil -(\mathbf B^{<1}_{X'})\rceil$ is 
effective and $p$-exceptional, 
$\mathcal J_{\Nglc}$ is an ideal sheaf on $X$ and 
$\mathcal J_{\Nglc}=
p_*\mathcal O_{X'}\left(-\lfloor \mathbf B^{>1}_{X'}\rfloor\right)$ 
holds. By definition, $\mathcal J_{\Nglc}$ is a generalization of 
non-lc ideal sheaves. 
\end{step}
\begin{step}\label{q-step1.3.3} 
In this step, we will check the inclusions $\mathcal J_{\Ngklt}\subset 
\mathcal I_{\Nqklt(X, \omega)}$ and $\mathcal J_{\Nglc}\subset 
\mathcal I_{\Nqlc(X, \omega)}$. 

We note that $\mathcal I_{\Nqklt(X, \omega)}=
f_*\mathcal O_Y(-\lfloor B_Y\rfloor)$ holds 
(see \cite[Propositions 6.3.1 and 
6.3.2]{fujino23} and the proof of \cite[Theorem 6.3.5 (i)]{fujino23}). 
Let $P$ be an irreducible component of $\Supp B^{\geq 1}_Y$. 
We take a suitable birational modification $p\colon X'\to X$ 
satisfying (i)--(v) and consider 
the induced basic $\mathbb R$-slc-trivial fibration $f'\colon (Y', B_{Y'})
\to X'$ (see \cite[Definition 3.4]{fujino-hashizume2}). 
We have the following commutative diagram. 
$$
\xymatrix{
   (Y', B_{Y'}) \ar[r]^-q \ar[d]_-{f'} & (Y, B_Y)\ar[d]^-f \\
   X' \ar[r]_-p & X
} 
$$ 
Note that $f'\colon (Y', B_{Y'})\to X'$ is a 
basic $\mathbb R$-slc-trivial fibration with $K_{Y'}+B_{Y'}=
q^*(K_Y+B_Y)$ such that $f'\colon (Y', B_{Y'})\to X'$ coincides with 
the base change of $f\colon (Y, B_Y)\to X$ by $p\colon X'\to X$ over 
some nonempty Zariski open subset of $X'$. 
Without loss of generality, by using the flattening theorem 
(see \cite[Th\'eor\`eme (5.2.2)]{raynaud-g}), 
we may assume that the image of $P':=q^{-1}_*P$ 
by $f'$ is a prime divisor $Q$ on $X'$. 
We put $\coeff _PB_Y=1+a$ with $a\geq 0$ and $\coeff _{P'}f'^*Q=m>0$. 
Then $\coeff _Q\mathbf B_{X'}=1-t$ with $t\leq -\frac{a}{m}$. 
By Lemma \ref{q-lem6.1}, 
$$
m\left(1+\lfloor -t\rfloor\right)\geq 1+\lfloor a\rfloor. 
$$ 
This means that if $$h\in \Gamma (U, \mathcal J_{\Ngklt})=\Gamma \left(U, 
p_*\mathcal O_{X'}\left(-\lfloor 
\mathbf B^{\geq 1}_{X'}\rfloor\right)\right)$$ then 
$$f^*h\in \Gamma \left(f^{-1}(U), 
\mathcal O_Y\left(-\lfloor B^{\geq 1}_Y\rfloor\right)\right)$$ 
for any Zariski open subset $U$ of $X$. 
Therefore, we obtain the desired inclusion 
$$\mathcal J_{\Ngklt}\subset f_*\mathcal O_Y(-\lfloor B_Y\rfloor)=
\mathcal I_{\Nqklt(X, \omega)}. 
$$ 
This is what we wanted. 
The same argument as above works for $\mathcal J_{\Nglc}$ and 
$\mathcal I_{\Nqlc(X, \omega)}=f_*\mathcal O_Y(-\lfloor B_Y
\rfloor+B^{=1}_Y)$. 
Hence we obtain the desired inclusion 
$\mathcal J_{\Nglc}\subset \mathcal I_{\Nqlc(X, \omega)}$. 
\end{step} 
We finish the proof of Theorem \ref{q-thm1.3}. 
\end{proof}

\section{Proof of Theorem \ref{q-thm1.1}}\label{q-sec7}

In this section, we will prove Theorem \ref{q-thm1.1} and 
Corollary \ref{q-cor1.2}. 
Our proof of Theorem \ref{q-thm1.1} here, 
which is a combination of Kawamata's X-method 
with Theorem \ref{q-thm1.3} 
in the framework of quasi-log schemes, is completely 
different from the proof given in \cite{fujino20}. 
A key idea of the proof of Theorem \ref{q-thm1.1} below is due to 
the argument in 
\cite{fujino-liu-haidong2}. 
Let us prove Theorem \ref{q-thm1.1}. 

\begin{proof}[Proof of Theorem \ref{q-thm1.1}]
We divide the proof into several small steps. 
\setcounter{step}{0}
\begin{step}\label{q-step1.1.1}
If $\dim X\setminus X_{-\infty}=0$, then Theorem \ref{q-thm1.1} 
obviously holds true. 
From now on, we assume that Theorem \ref{q-thm1.1} holds for 
any quasi-log scheme $Z$ with 
$\dim Z\setminus Z_{-\infty}<
\dim X\setminus X_{-\infty}$ by induction on $\dim X\setminus X_{-\infty}$. 
\end{step}
\begin{step}\label{q-step1.1.2}
We take a qlc stratum $W$ of $[X, \omega]$. 
We put $X'=W\cup X_{-\infty}$. 
Then, by adjunction (see Theorem \ref{q-thm5.5} (i)), 
$X'$ has a natural quasi-log scheme 
structure induced by $[X, \omega]$. 
By the vanishing theorem (see Theorem \ref{q-thm5.5} (ii)), 
we have 
$$R^1\pi_*
(\mathcal I_{X'}\otimes \mathcal O_X(mL))=0$$ for 
every $m\geq q$, where $\mathcal I_{X'}$ is the defining 
ideal sheaf of $X'$ on $X$. 
Therefore, we obtain that the restriction map  
$$\pi_*\mathcal O_X(mL)\to \pi_*\mathcal O_{X'}
(mL)$$ is surjective for 
every $m\geq q$. 
Thus, we may assume that $\overline {X\setminus X_{-\infty}}$ 
is irreducible for the proof of 
Theorem \ref{q-thm1.1} 
by the following commutative 
diagram. 
$$
\xymatrix{
\pi^*\pi_*\mathcal O_X(mL)\ar[r]\ar[d]
&\pi^*\pi_*\mathcal O_{X'}(mL)\ar[r]\ar[d]&0\\
\mathcal O_X(mL)\ar[r]&\mathcal O_{X'}(mL)\ar[r]&0}
$$
\end{step} 
\begin{step}\label{q-step1.1.3}
We put $W=\overline {X\setminus X_{-\infty}}$. 
By Step \ref{q-step1.1.2}, $W$ is irreducible. 
In this step, we further assume that $W\cap \Nqklt(X, \omega)=\emptyset$. 
In this case, $W$ is a normal variety 
(see \cite[Lemma 6.3.9]{fujino23}) 
and $[W, \omega|_W]$ 
is a quasi-log canonical pair (see \cite[Lemma 6.3.12]{fujino23}). 
By \cite[Theorem 1.7]{fujino30} 
(see also Theorem \ref{q-thm1.3}), 
there exists a projective 
birational morphism $p\colon W'\to W$ from a 
smooth quasi-projective variety $W'$ such that 
$$
p^*\omega|_W=K_{W'}+B_{W'}+M_{W'}, 
$$ 
$\Supp B_{W'}$ is a simple normal crossing divisor, 
$B_{W'}=B^{<1}_{W'}$, $\Supp B^{<0}_{W'}$ is $p$-exceptional, 
and $M_{W'}$ is a potentially nef $\mathbb R$-divisor on $W'$. Hence 
$q(p^*L|_W)-(K_{W'}+B_{W'})$ is nef and big over $S$ and 
$\lceil -B_{W'}\rceil$ is effective and $p$-exceptional. 
Then, by Theorem \ref{p-thm4.2}, 
$\mathcal O_{W'}(mp^*L|_W)$ is $\pi\circ p$-generated for 
every $m\gg 0$. 
Hence $\mathcal O_W(mL|_W)$ is $\pi$-generated for every $m\gg 0$. 
Since $W\cap X_{-\infty}=\emptyset$ by assumption, 
$\mathcal O_X(mL)$ is $\pi$-generated 
for every $m\gg 0$. 
Therefore, from now on, we may assume that 
$W\cap \Nqklt(X, \omega)\ne \emptyset$. 
\end{step}
\begin{step}\label{q-step1.1.4} 
By \cite[Lemma 4.19]{fujino35}, 
$[W, \omega|_W]$ has a natural quasi-log scheme structure 
induced by $[X, \omega]$ such that 
$\mathcal I_{\Nqklt(W, \omega|_W)}=\mathcal I_{\Nqklt(X, \omega)}$ 
holds. 
Let $\nu\colon Z\to W$ be the normalization. 
Then, by \cite[Theorem 1.9]{fujino35} (see also \cite{fujino-liu-haidong1}), 
there exists a proper surjective morphism 
$f'\colon Y'\to Z$ from a quasi-projective globally embedded simple 
normal crossing pair $(Y', B_{Y'})$ such that every stratum of $Y'$ is 
dominant onto $Z$ and that 
$$
\left(Z, \nu^*(\omega|_W), f'\colon (Y', B_{Y'})\to Z\right)
$$ 
naturally becomes a quasi-log scheme with 
$$
\nu_*\mathcal I_{\Nqklt(Z, \nu^*(\omega|_W))}=
\mathcal I_{\Nqklt(W, \omega|_W)}=\mathcal I_{\Nqklt(X, \omega)}. 
$$ 
By Theorem \ref{q-thm1.3}, we can take a projective 
birational morphism 
$p\colon Z'\to Z$ from a smooth quasi-projective variety $Z'$ 
with $$\mathbf K_{Z'}+\mathbf B_{Z'}
+\mathbf M_{Z'}=p^*\nu^*(\omega|_W)$$  satisfying 
(i)--(v) in Theorem \ref{q-thm1.3} such that the following inclusion 
$$
\mathcal J_{\Ngklt}=p_*\mathcal O_{Z'}(-\lfloor \mathbf B_{Z'}\rfloor)
\subset \mathcal I_{\Nqklt(Z, \nu^*(\omega|_W))}
$$ 
holds. 
\end{step}
\begin{step}\label{q-step1.1.5}
By induction on $\dim W$ (see Step \ref{q-step1.1.1}) 
or the assumption that $\mathcal O_{X_{-\infty}}(mL)$ 
is $\pi$-generated for every $m\gg 0$, 
$\mathcal O_{\Nqklt(X, \omega)}
(mL|_{\Nqklt(X, \omega)})$ is $\pi$-generated for every $m\gg 0$. 
As in Step \ref{q-step1.1.2}, 
the restriction map 
$$
\pi_*\mathcal O_X(mL)\to \pi_*\mathcal O_{\Nqklt(X, \omega)}
(mL|_{\Nqklt(X, \omega)})
$$ is surjective for $m\geq q$ since 
$R^1\pi_*(\mathcal I_{\Nqklt(X, \omega)}
\otimes \mathcal O_X(mL))=0$ 
for $m\geq q$ by the vanishing theorem 
(see Theorem \ref{q-thm5.5} (ii)). 
Therefore, the relative base locus of $\mathcal O_X(mL)$ is disjoint from 
$\Nqklt(X, \omega)$ for every $m\gg 0$. 
\end{step}
\begin{step}\label{q-step1.1.6} 
We take a finite affine open covering $S=\bigcup _{\lambda
\in \Lambda} U_\lambda$ of $S$. It is sufficient to 
prove this theorem on each affine open subset $U_\lambda$. 
Hence, by replacing $S$ with $U_\lambda$, 
we may further assume that $S$ is affine. 
Let $\ell$ be a sufficiently large prime number. 
Then $\Bs|\ell L| \cap \Nqklt(Z, \nu^*(\omega|_W))=
\Bs|\ell L|\cap p(\mathbf B^{\geq 1}_{Z'})=\emptyset$ by Step 
\ref{q-step1.1.5}. 
We note that $\mathbf B^{<0}_{Z'}$ is $p$-exceptional 
and that $\mathbf M_{Z'}$ is nef over $S$. 
By Lemma \ref{p-lem4.4}, 
$$
H^0\left(Z, \mathcal O_Z(\ell^s\nu^*(L|_W))\otimes 
p_*\mathcal O_{Z'}(-\lfloor \mathbf B_{Z'}\rfloor))\otimes 
\mathcal O_Z\to \mathcal O_Z(\ell ^s \nu^*(L|_W)\right)
$$ is surjective on $Z\setminus p(\mathbf B^{\geq 1}_{Z'})$ 
for some positive integer $s$. 
We note that 
\begin{equation*}
\begin{split}
&H^0\left(Z, \mathcal O_Z(m\nu^*(L|_W))
\otimes p_*\mathcal O_{Z'}(-\lfloor 
\mathbf B_{Z'}\rfloor)\right)\\ 
&\subset 
H^0\left(Z, \mathcal O_Z(m\nu^*(L|_W))\otimes \mathcal I_{\Nqklt(Z,  
\nu^*(\omega|_W))}\right)\\
&=
H^0\left(W, \mathcal O_W(mL|_W)\otimes 
\mathcal I_{\Nqklt(W,  \omega|_W)}\right)\\
&=
H^0\left(X, \mathcal O_X(mL)\otimes \mathcal I_{\Nqklt(X,  \omega)}\right)
\end{split}
\end{equation*}
holds for every integer $m$. 
Therefore, $\Bs|\ell ^s L|\cap (X\setminus \Nqklt(X, \omega))=\emptyset$. 
This implies that $\Bs|\ell ^s L|=\emptyset$. 
We take a sufficiently large prime number $\ell '$ with $\ell '\ne \ell$. 
By the same argument as above, we can find a positive integer $s'$ such that 
$\Bs |\ell'^{s'}L|=\emptyset$. 
Without loss of generality, we may assume that 
$\ell^s<\ell'^{s'}$ holds. 
Note that $\gcd (\ell^s, \ell'^{s'})=1$ since 
$\ell \ne \ell'$. 
We put $m_0=\ell^s \left(\ell'^{s'} 
-\left\lceil \frac{\ell'^{s'}}{\ell^s}\right\rceil\right)$. 
By Lemma \ref{p-lem4.3}, for every positive integer $m$ with 
$m\geq m_0$, there exist nonnegative integers $u$ and $v$ such that 
$m=u\ell^s+v\ell'^{s'}$. 
This means that 
$\Bs |mL|=\emptyset$ for every $m\geq m_0$ since 
$\Bs|\ell^sL|=\Bs|\ell'^{s'}L|=\emptyset$. 
\end{step} 
We finish the proof of Theorem \ref{q-thm1.1}. 
\end{proof}

We close this section with the proof of Corollary \ref{q-cor1.2}. 

\begin{proof}[Proof of Corollary \ref{q-cor1.2}]
Let $(X, \Delta)$ be a log canonical pair. 
We put $\omega=K_X+\Delta$. 
Then $[X, \omega]$ naturally becomes a quasi-log scheme 
with $\Nqlc(X, \omega)=\emptyset$ (see Example 
\ref{q-ex5.3}). 
By assumption, $qL-\omega$ is nef and log big over 
$S$ with respect to $[X, \omega]$. 
Hence $\mathcal O_X(mL)$ is $\pi$-generated 
for every $m \gg 0$ by Theorem \ref{q-thm1.1}. 
\end{proof}

\section{Comments on Ambro's paper:~Quasi-log varieties}\label{q-sec8}

In this section, we make many comments on \cite{ambro} 
to help the reader 
understand differences between Ambro's original 
approach in \cite{ambro} and our framework of quasi-log schemes 
mainly discussed in \cite[Chapter 6]{fujino23}. 
Note that \cite{fujino11} is a gentle introduction to the 
theory of quasi-log varieties. We also note that 
a {\em{quasi-log variety}} in \cite{ambro} and \cite{fujino11} 
is called a {\em{quasi-log scheme}} in the author's recent 
papers because it may have non-reduced components. 

\setcounter{subsection}{-1}
\subsection{Introduction}\label{q-subsec8.0}
By \cite[Definition 1]{ambro}, 
a {\em{generalized log variety}} $(X, B)$ is a pair consisting of 
a normal variety $X$ and an effective $\mathbb R$-divisor 
$B$ on $X$ such that $K_X+B$ is $\mathbb R$-Cartier. 
Note that we sometimes call $(X, B)$ a {\em{normal pair}} in 
some literature. 
When $(X, B)$ is log canonical, it is called a {\em{log variety}}. 
The main result of \cite{ambro} is the cone and contraction theorem for 
generalized log varieties (see \cite[Theorem 2]{ambro}). 
In order to establish the cone and contraction theorem, 
Ambro introduced the notion of {\em{quasi-log varieties}}, 
which was motivated 
by Kawamata's X-method. He also said that 
the motivation behind 
\cite{ambro} is Shokurov's idea that log varieties 
and non-kawamata log terminal loci should be treated on an equal 
footing. In \cite{fujino14}, we recovered \cite[Theorem 2]{ambro} without 
using the framework of quasi-log schemes. 
The approach in \cite{fujino14} was influenced by not only 
Kawamata's X-method but also the theory of (algebraic) 
multiplier ideal 
sheaves (see \cite{lazarsfeld}). 
We recommend the reader who is interested only in the 
cone and contraction theorem for generalized log varieties to 
see \cite{fujino14}, which seems to be more accessible than 
\cite{ambro}. 
We note that \cite{fujino7} and \cite{fujino13} may help the 
reader understand \cite{fujino14}. 
The cone and contraction theorem for generalized log varieties 
(see \cite[Theorem 2]{ambro} and \cite[Theorem 1.1]{fujino14}) 
plays a crucial role in the minimal model theory for 
log surfaces (see \cite{fujino15}, \cite{fujino19}, \cite{fujino32}, 
\cite{fujino33}, and \cite{fujino-tanaka}). 

\subsection{Section 1.~Preliminary}\label{q-subsec8.1} 
In Section 1 in \cite{ambro}, some standard definitions are collected. 
In the theory of quasi-log schemes, we usually treat highly singular 
reducible schemes. Moreover, we have to treat $\mathbb R$-Cartier 
divisors and $\mathbb R$-line bundles. 
We note that Kleiman's famous ampleness criterion does not 
necessarily hold true for singular complete non-projective schemes 
(see \cite{fujino5} and \cite{fujino32}). 
On the other hand, the Nakai--Moishezon ampleness criterion 
for $\mathbb R$-line bundles holds for any complete schemes 
(see \cite{fujino-miyamoto2}). It sometimes 
may be very useful when we treat $\mathbb R$-line 
bundles on highly singular schemes. 
Moreover, the Nakai--Moishezon ampleness criterion on 
complete algebraic spaces is crucial for the proof of 
the projectivity of some moduli spaces (see \cite{fujino28}). 

\subsection{Section 2.~Normal crossing pairs} \label{q-subsec8.2} 
Ambro defined {\em{multicrossing singularities}} and 
considered their associated hypercoverings (see 
\cite[Definition 2.1 and Lemma 2.2]{ambro}). 
Moreover, he defined 
{\em{multicrossing divisors}}. Then 
he finally introduced the notion of ({\em{embedded}}) 
{\em{normal crossing pairs}} 
(see \cite[Definitions 2.3 and 2.7]{ambro}). 
He used embedded normal crossing pairs to define quasi-log 
varieties. On the other hand, our framework of the theory of 
quasi-log schemes in \cite[Chapter 6]{fujino23} 
uses the notion of 
{\em{globally embedded simple normal crossing pairs}}. 
Note that 
\cite[Propositions 6.3.1, 6.3.2, and 6.3.3]{fujino23} 
is much more flexible than 
\cite[Proposition 2.8]{ambro}. 
We think that our approach is more accessible than Ambro's 
because {\em{globally embedded simple normal crossing pairs}} are much 
easier to treat than 
{\em{embedded normal crossing pairs}}. 
We can use the standard techniques in the 
theory of minimal models for higher-dimensional 
algebraic varieties to treat globally 
embedded simple normal crossing pairs. 
In general, simple normal crossing divisors 
behave much better than normal crossing divisors 
(see \cite{fujino6}). 

\subsection{Section 3.~Vanishing theorems}\label{q-subsec8.3}
Section 3 in \cite{ambro} 
is a short section on injectivity, vanishing, and torsion-free theorems. 
The proof of \cite[Theorems 3.1 and 3.2]{ambro} 
is hard to follow. In \cite[Chapter 5]{fujino23} 
(see also \cite{fujino7}, \cite{fujino14}, \cite{fujino22}, 
\cite{fujino27}, and so on), we 
give a rigorous proof of \cite[Theorems 3.1 and 3.2]{ambro} 
and 
treat some more general results. 
Our approach is slightly different from Ambro's and 
is based on the theory of mixed Hodge structures 
on cohomology with compact support. 
A survey article \cite{fujino26} may help the reader understand 
our approach to vanishing theorems. 
The reader can find some related vanishing theorems in 
\cite{fujino17}, \cite{fujino18}, \cite{fujino29}, and so on. 

\subsection{Section 4.~Quasi-log varieties}\label{q-subsec8.4} 
As we mentioned above, a quasi-log variety is called a quasi-log scheme 
in \cite{fujino23}. 
Section 4 is the main part of \cite{ambro}. 
In Section 4, Ambro defined quasi-log varieties 
in \cite[Definition 4.1]{ambro}. 
Ambro's definition is slightly different from 
ours in \cite[Definition 6.2.2]{fujino23}. 
Note that $(Y, B_Y)$ in \cite[Definition 4.1]{ambro} is an 
embedded normal crossing pair and $(Y, B_Y)$ in 
\cite[Definition 6.2.2]{fujino23} is a globally embedded simple 
normal crossing pair. 
For the details of this difference, see \cite{fujino24}. 
The most important result in this section is \cite[Theorem 4.4]{ambro}, 
which is {\em{adjunction}} and {\em{vanishing}} for quasi-log 
varieties. In \cite[Section 6.3]{fujino23}, we prepare some useful propositions 
(see \cite[Propositions 6.3.1, 6.3.2, and 6.3.3]{fujino23}) and 
prove adjunction and vanishing for quasi-log schemes in \cite[Theorem 6.3.5]
{fujino23}. We also discuss some other basic properties of 
quasi-log schemes in \cite[Sections 6.3 and 6.4]{fujino23}. 
Note that a {\em{qlc center}} in \cite{ambro} is called 
a {\em{qlc stratum}} in \cite[Chapter 
6]{fujino23}. We also note that a {\em{qlc center which is not maximal with respect 
to the inclusion}} in \cite{ambro} is called a {\em{qlc center}} 
in \cite[Chapter 6]{fujino23}. Hence $\mathrm{LCS}(X)$ in 
\cite[Definition 4.6]{ambro} is nothing but $\Nqklt(X, \omega)$ 
in \cite[Chapter 6]{fujino23}. We can recover \cite[Proposition 4.7]{ambro} 
by \cite[Lemma 6.3.9]{fujino23}. Our proof seems to be simpler. 
Note that \cite[Proposition 4.8]{ambro} is \cite[Theorem 6.3.11]{fujino23}. 
In \cite{fujino-hashizume2} (see 
also \cite{fujino-hashizume3}), we completely 
generalize \cite[Theorem 4.9]{ambro}. 
Our treatment depends on the theory of variations of mixed Hodge 
structure on cohomology with compact support 
(see \cite{fujino-fujisawa1}, \cite{fujino-fujisawa-saito}, 
\cite{fujino30}, and \cite{fujino-fujisawa-liu}). 
On the other hand, \cite[Theorem 4.11]{ambro} only 
uses the theory of variations of pure Hodge structure. 

\subsection{Section 5.~The cone theorem}\label{q-subsec8.5}
In \cite[Section 5]{ambro}, the cone and contraction theorem 
for quasi-log schemes was established in full generality. 
The results in \cite[Section 5]{ambro} are recovered in 
\cite[Sections 6.5, 6.6, and 6.7]{fujino23}. 
Although we slightly changed and improved some arguments, 
the treatment in \cite[Sections 6.5, 6.6, and 6.7]{fujino23} 
is essentially the same as that in \cite[Section 5]{ambro}. 
The basepoint-free theorem for quasi-log 
schemes 
(see \cite[Theorem 5.1]{ambro} 
and \cite[Theorem 6.5.1]{fujino23}) 
is generalized in various directions (see \cite{fujino8}, 
\cite{fujino10}, \cite{fujino21}, \cite{fujino25}, \cite{fujino34}, 
\cite{fujino-liu-haidong2}, 
\cite{fujino-miyamoto1}, and \cite{liu-haidong}). 
The theory of quasi-log schemes gives a very powerful 
framework for basepoint-freeness 
(see also the proof of Theorem \ref{q-thm1.1} in 
Section \ref{q-sec7} in this paper). 

\subsection{Section 6.~Quasi-log Fano contractions}\label{q-subsec8.6}
In \cite[Section 6]{ambro}, Ambro 
specialized some results in \cite[Section 5]{ambro} 
for quasi-log Fano contraction morphisms. 
In \cite[Section 6.8]{fujino23}, we treat ({\em{relative}}) 
{\em{quasi-log Fano schemes}}. Moreover, in 
\cite{fujino-liu-wenfei}, \cite{fujino31}, \cite{fujino35}, 
and \cite{fujino-hashizume1}, 
we 
discuss simply connectedness, rationally 
chain connectedness, lengths of rational curves 
for (relative) quasi-log Fano schemes. 
In \cite{fujino31} and \cite{fujino35}, 
we use not only quasi-log schemes but also 
some results obtained by the theory of basic slc-trivial fibrations. 
Moreover, in \cite{fujino-hashizume1}, we also use the minimal 
model program for log canonical pairs. 
Hence the results in \cite{fujino31}, \cite{fujino35}, 
and \cite{fujino-hashizume1} are much more general than 
those in \cite[Section 6]{ambro}. 

\subsection{Section 7.~The log big case}\label{q-subsec8.7}
In Section 7, Ambro replaced the 
ampleness in some theorems with the {\em{nef and log bigness}}. 
Note that \cite[Theorem 7.2]{ambro} is the basepoint-free 
theorem of Reid--Fukuda type for quasi-log schemes, which is 
Theorem \ref{q-thm1.1} in this paper. 
In \cite{ambro}, there is no detail of the proof of \cite[Theorem 7.2]{ambro}. 
Now we have a rigorous proof of \cite[Theorem 7.2]{ambro}. 
The reader can find vanishing theorems for nef and log big 
divisors in \cite[Theorem 5.8.2 and Theorem 6.3.5 (ii)]{fujino23}. 
We note that \cite[Theorem 6.3.8]{fujino23} is a slight generalization of 
\cite[Theorem 7.3]{ambro}. We know that 
everything in \cite[Section 7]{ambro} holds true. 
We note that 
the proof of Theorem \ref{q-thm1.1} in this paper 
depends on some deep results in the theory of 
variations of mixed Hodge structure on cohomology with 
compact support (see \cite{fujino-fujisawa1}, 
\cite{fujino-fujisawa-saito}, \cite{fujino30}, 
\cite{fujino35}, \cite{fujino-hashizume2}, and so on). 

%%%%%%%%%%%%%%%

\end{document}